\documentclass[a4paper]{amsart}

\usepackage{comment}
\usepackage{graphicx}
\usepackage{amsmath,amsthm,amssymb,bbm}
\usepackage[all]{xy}
\usepackage{tikz-cd}

\DeclareMathOperator{\Ker}{Ker}                  
\renewcommand{\phi}{\varphi}                  
\newcommand{\eps}{\varepsilon}                  

\newcommand{\ind}{\mathbbm{1}}
\newcommand{\coker}{\text{\rm coker}\,}

\newcommand{\R}{\mathbb{R}}
\newcommand{\Z}{\mathbb{Z}}

\newcommand{\N}{\mathbb{N}}

\newtheorem{Theorem}{Theorem}

\newtheorem{Proposition}[Theorem]{Proposition}

\newtheorem{Example}[Theorem]{Example}

\newtheorem{Definition}[Theorem]{Definition}

\begin{document}

\title{$K$-theory of the Chair Tiling via $AF$-algebras}

\author[A.\ Julien]{Antoine Julien}
\address{Department of Mathematical Sciences, NTNU, Trondheim, Norway}
\email{antoine.julien@math.ntnu.no}

\author[J.\ Savinien]{Jean Savinien}
\address{Universit\'e de Lorraine, Institut Elie Cartan de Lorraine, UMR 7502, Metz, F-57045, France}
\address{CNRS, Institut Elie Cartan de Lorraine, UMR 7502, Metz, F-57045, France.}
\email{jean.savinien@univ-lorraine.fr}

\begin{abstract}
{We compute the $K$-theory groups of the groupoid C$^\ast$-algebra of the chair tiling, using a new method.
We use exact sequences of Putnam to compute these groups from the $K$-theory groups of the $AF$-algebras of the substitution and the induced lower dimensional substitutions on edges and vertices.}
\end{abstract}


\maketitle

\section{Introduction}

A repetitive, aperiodic tiling (or point-set) exhibits two \emph{a priori} antagonistic behaviours. On the one hand, local configurations of any finite size repeat; on the other hand, the way in which they repeat is not predictable, in the sense that it can't be described by a periodic lattice. The study of such repetitive, aperiodic tilings is the basis of the theory of \emph{aperiodic order}.
This theory gained a lot of traction in the 1980's, when it appeared that such objects--such as the Penrose tilings--could be used to model quasicrystals in nature (see~\cite{BG13} for a recent treatment).

A $\R^d$-dynamical system $(\Omega, \R^d)$ can be naturally associated with a given aperiodic tiling $T$: it is a compact space containing all tilings which ``look locally like $T$''. A close analogue in the symbolic setting is the $\Z^d$-subshift associated with a given word in $\{0,1\}^{\Z^d}$.
The tiling space $\Omega$ is then the perfect analogue of the \emph{suspension} of the subshift.

It is very relevant, especially in the context of crystallography, to study the $C^*$-algebra $C(\Omega) \rtimes \R^d$, associated with the groupoid $\Omega \rtimes \R^d$ of $\R^d$ acting on $\Omega$.\footnote{Often, it is technically simpler to study a Morita-equivalent $C^*$-algebra corresponding to the reduction of $\Omega \rtimes \R^d$ on a transversal.}
In particular, the ordered $K_0$-group of this $C^*$-algebra can be related to the gaps in the spectrum of a Schr\"odinger operator describing the motion of electrons on a quasicrystal~\cite{BHZ00}.
There have been many approaches for computing topological invariants of tiling spaces and their $C^*$-algebras, and we very partially decide to cite Kellendonk~\cite{Kel97} (one of the first approaches, relating $K$-theory to the group of coinvariants), Moustafa~\cite{Mou10} ($K$-theory computations for the Pinwheel tiling, involving the explicit construction of fiber bundles representing $K$-elements), and Oyono-Oyono--Petite~\cite{OOP11} (very sophisticated computations involving the $K$-theory of the hyperbolic Penrose tiling).
Beyond particular examples, we can single out two important families of aperiodic, repetitive tilings: self-similar tilings and cut-and-project tilings (also known as model sets). For each of these, techniques have been developed to compute topological invariants~\cite{AP98,GHK13}.
However, these methods rely in general on the Thom--Connes isomorphism $K_*(C(\Omega) \rtimes \R^d) \simeq K_*(C(\Omega)) \simeq K^*(\Omega)$ (with a possible grading shift depending on the parity of $d$), and on the Chern isomorphism
\[
 K^* (\Omega) \otimes \mathbb Q = \bigoplus_i \check H^i (\Omega) \otimes \mathbb Q,
\]
with appropriate grading depending on the parity of $i$.
In any case, the computation of $K$ is often reduced to the computation of the cohomology groups of $\Omega$ (in one or another form: there are \emph{a lot} of ways to describe the cohomology of tiling spaces).

A very notable exception is the case of ``tilings'' of dimension one. Assume (without loss of generality) that the one dimensional tiling space is the suspension of an aperiodic, minimal subshift $(X,\sigma)$. The $C^*$-algebra of relevance is, in this case, $C(X) \rtimes_\sigma \Z$ (up to Morita-equivalence). It was established by Putnam that
\[
 K_0 (C(X) \rtimes \Z) \simeq K_0 \bigl( \langle C(X), C_0(X\setminus\{y\}) u \rangle \bigr),
\]
for some $y\in X$, and where $u$ is the unitary in $C(X) \rtimes \Z$ implementing the action. This is especially interesting, because the right-hand algebra is an $AF$-algebra, and its $K_0$-group can be computed by an appropriate direct limit.
The philosophy in this case is that the tiling algebra contains a large $AF$-subalgebra, and if this $AF$-algebra is large enough, their $K_0$-groups are isomorphic.

In higher dimensions, such an approach cannot succeed. For dimension two and higher, there is no such thing as a ``maximal $AF$-subalgebra'' in a tiling algebra.
However, if the tiling algebra contains a big enough subalgebra, we can hope: 1) that the $K$-groups of the subalgebra are easier to compute; and 2) that it is possible to compute the difference between the $K$-groups of these two-algebras.
The basic step of such an approach was presented by Putnam in a series of two papers~\cite{Put97, Put98}. These papers compare the $K$-theory of a groupoid algebra and of a ``disconnected'' subgroupoid algebra.

In the present paper, we present a proof-of-concept that these methods can actually be used to compute $K$-groups of tiling algebras.
This has several advantages. First, it is aesthetically pleasing to compute the $K$-groups of tiling algebras without leaving the non-commutative setting.
Second, it makes it easier to identify the generators of these $K$-groups. This can be especially important for working out the order-structure, or the image under a trace of these generators. It is known (\emph{gap-labeling theorem}) that the image under the trace of the $K_0$-group of a tiling algebra is equal to the module of frequencies.
However, if $\omega$ is a $S^1$-valued groupoid $2$-cocycle on $\Omega \rtimes \R^d$, the image under the trace of the $K_0$-elements of the \emph{twisted} $C^*$-algebra $C(\Omega) \rtimes^\omega \R^d$ can be more complicated.
In this case, the computation of the gap labeling group will rely on a precise identification of the generators of $K_0$.
Such twisted algebra appear naturally when describing the quantum Hall effect on aperiodic solids~\cite{BESB94} (see also more recently~\cite{BM15}).
Besides, the problem of computing the image under the trace of the $K_0$-group in the twisted setting also appeared recently in time-frequency analysis for describing Gabor systems associated with quasiperiodic point-sets~\cite{Kre16}.

Heuristically, our computations of the $K$-theory groups of the chair tiling appears to be guided by a filtration of the groupoid $\Omega \rtimes \R^d$  (or rather of its classifying space $\Omega$).
In the commutative case, Barge--Diamond--Hunton--Sadun~\cite{BDHS10} have produced a machinery to compute the cohomology of a tiling space.
They compute the cohomology of $\Omega$ with exact sequences arising from a filtration based on the geometry of the tiles (vertices, edges, and faces), and use relative cohomology theory.
The computations presented in the present paper can be viewed as a non commutative analogue of their approach for $K$-theory.

Finally, the approach presented here leads to an interesting question: it appears that the $K$-invariants of the tiling $C^*$-algebra can be recovered entirely by computing dimension groups (i.e.\@ $K_0$-groups of $AF$-algebras) and connecting maps. Previously, the authors investigated how a tiling groupoid (which can be seen as an equivalence relation if the tiling has no period) can be described entirely by reading tail-equivalence relations from Bratteli diagrams, together with additional ``adjacency'' information~\cite{BJS12,JS12}.
The next question is whether the same approach can provide some new insight on the structure of the algebras themselves.
Is there a way to describe a tiling algebra (or a Cantor minimal $\Z^d$-system---of which tiling algebras are tractable examples) as being built from $AF$-algebras ``glued'' together in an appropriate way?

\medskip

The paper is organized as follows.
In Section~\ref{sec-Setup} we present our setup: define the chair substitution, the decorations and the boundary tilings, and spell out the various inverse limits describing the tiling space $\Omega$ and the space of boundary tilings.
In Section~\ref{sec-def} we define the groupoids we will use: the groupoid~$G$ of the chair tiling, the groupoids of Putnam (in particular those associated with boundary tilings), and the $AF$-groupoids.
We describe Putnam's exact sequences relating the $K$-theory of these various groupoids in Section~\ref{sec-PutES}, and spell out the $K$-theory maps explicitly.
We tackle the computations in Section~\ref{sec-Kth}: first we compute the $K_0$-groups of the $AF$-groupoids and give the explicit generators, and next compute the exact sequences.
We find:
\[
K_0 \bigl( C^\ast(G) \bigr)  \simeq \Z \bigl[ \frac{1}{4}\bigr] \oplus \Z \bigl[ \frac{1}{2}\bigr]^2 \oplus \Z,
\qquad 
K_1 \bigl( C^\ast(G) \bigr)  \simeq \Z \bigl[ \frac{1}{2}\bigr]^2.
\]
We used at several places the assistance of a computer algebra software. A more thorough explanation of some of the steps is given in appendix of the current version of this paper (but is not included in the published version). None of the computation uses specialized libraries: only basic functions such as eigenvalues, eigenvectors and Schmidt normal form are needed.

\section{Setup}
\label{sec-Setup}

\subsection{The chair substitution}

Consider the substitution rule $\omega_0$ given by the left-hand side of Figure~\ref{fig:subst-undecorated}. Given a set of prototiles, a substitution on this set is a process which inflates the support of a tile and covers it by copies of the prototiles.
In this example, we have four prototiles, one for each orientation. This substitution is often called the ``chair'' substitution, because it is equivalent (up to recoding) to the substitution pictured in Figure~\ref{fig:subst-undecorated} (right).
The substitution $\omega_0$ is defined on tiles, but it can be defined on finite or infinite sets of tiles by concatenation.

\begin{figure}[htp]
\begin{center}
\includegraphics[scale=0.65]{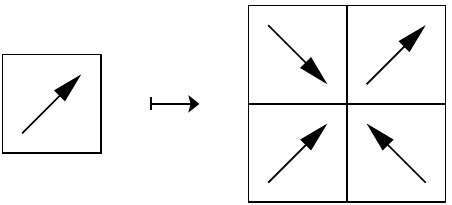} \qquad \includegraphics[scale=0.8]{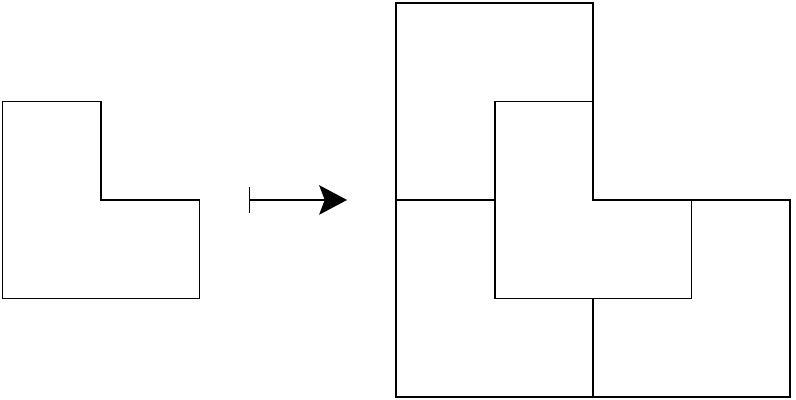}
\caption{{\small The ``arrow-chair'' substitution (left) and the original chair substitution. The substitution is shown on one orientation of tiles; the substitution of the other tiles is obtained by rotation.}}
\label{fig:subst-undecorated}
\end{center}
\end{figure}

A substitution tiling associated with this rule is a set of tiles $T$, which are all translates of the four prototiles, whose union covers $\R^2$, and such that:
\begin{itemize}
 \item for all finite configuration of tiles $P \subset T$ (we call $P$ a \emph{patch} of $T$), there exists a tile $t$ such that $P$ is a patch of $\omega_0^n(t)$ for some $n$ big enough.
\end{itemize}
A typical patch is pictured in Figure~\ref{fig:chair-patch}.

\begin{figure}[htp]
\begin{center}
\includegraphics[scale=0.8]{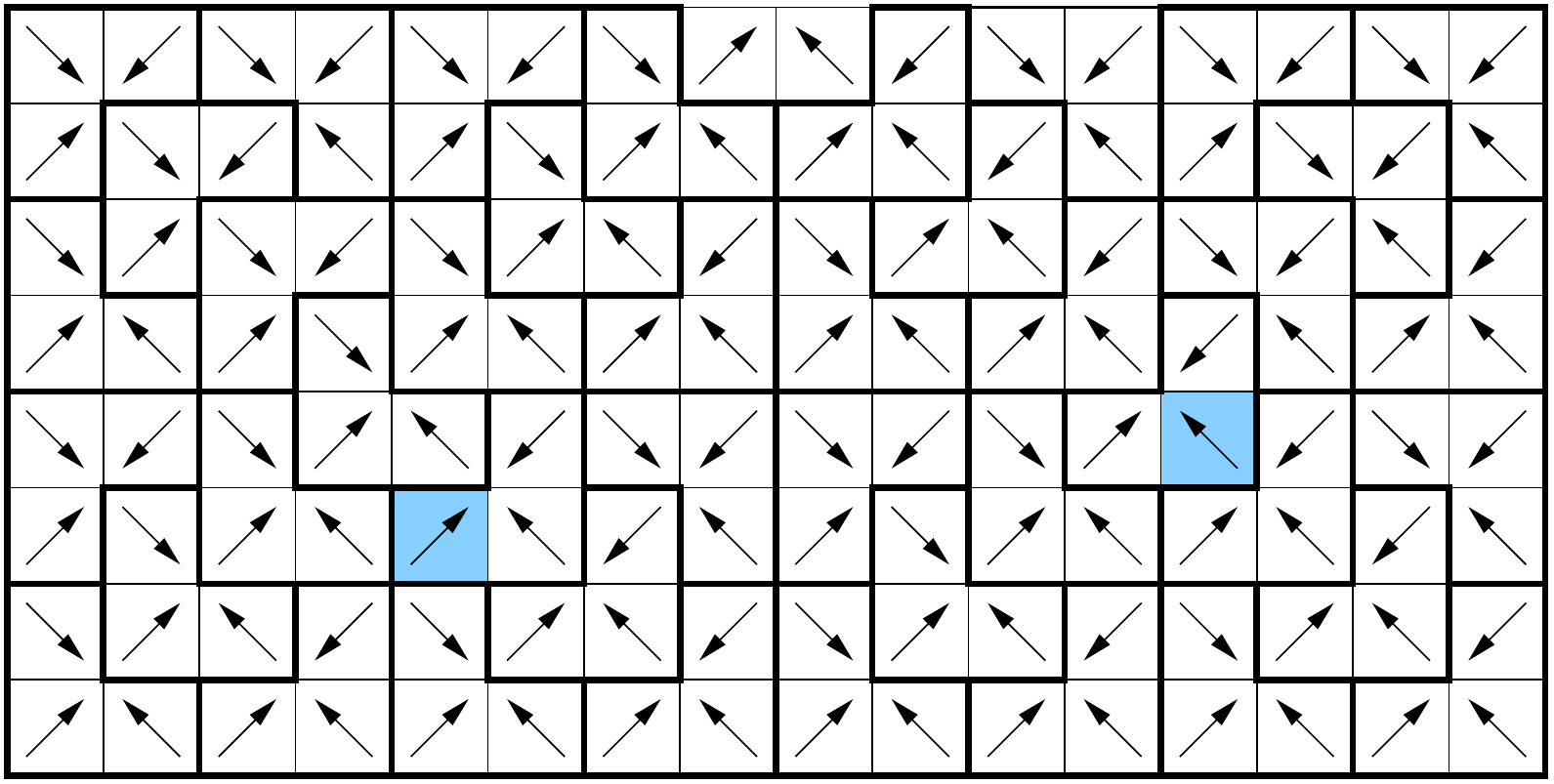}
\caption{{\small A patch of the chair tiling. Notice how arrowed tiles can be bundled up consistently into chairs. Conversely, the rule partitioning chairs into arrowed squares is well defined.
The left-hand side marked tile corresponds to a decorated tile of type $D_2$, and the right-hand side one corresponds to a decorated tile of type $rA$, where $r$ is the rotation of angle $\pi/2$ (see below).}}
\label{fig:chair-patch}
\end{center}
\end{figure}

Let $\Omega$ be the space of all substitution tilings associated with $\omega_0$. By convention, tilings are not identified up to translation: $T$ and $T-x$ are in general two different tilings.
The space $\Omega$ comes with the topology given by the distance:
\[
 d(T,T') < \eps \quad \text{if} \quad T-x, \ T'-x' \text{ agree on } B(0,1/\eps), \text{ with } \| x\|, \|x'\| < \eps.
\]
It is well known that it is a compact space, on which $\R^2$ acts freely and minimally by translation.

Our goal in this paper is to provide a method to compute the $K$-theory of the $C^*$-algebra $C(\Omega) \rtimes \R^2$.
This example has been treated before, but to the best of our knowledge, all methods used to compute the $K$-theory of the noncommutative algebra boil down to computing the cohomology of the space $\Omega$. We propose here a method which is in essence noncommutative.

\subsection{Tiles, boundaries and decorations}

The chair substitution allows one to define the space in a fairly simple way, but it will be necessary to use a ``decorated substitution''. Figure~\ref{fig-tiles} describes a substitution on 24 tiles (6 are pictured, and the other are given by rotation), which we denote by $\omega$.
The added data (decoration) is pictured in red and consists of information on some of the neighbouring tiles. 
We denote $A, B$, etc.\@ the tiles up to rotation, and $A^0$, $B^0$, etc.\@ the tiles as pictured with the north-east orientation. We let $r$ represent the rotation of angle $\pi/2$, so that the tiles can be written $A^0$, $rC_2^0$, $r^2D_2^0$, etc.

\begin{figure}[htp]
\begin{center}
\includegraphics[scale=0.65]{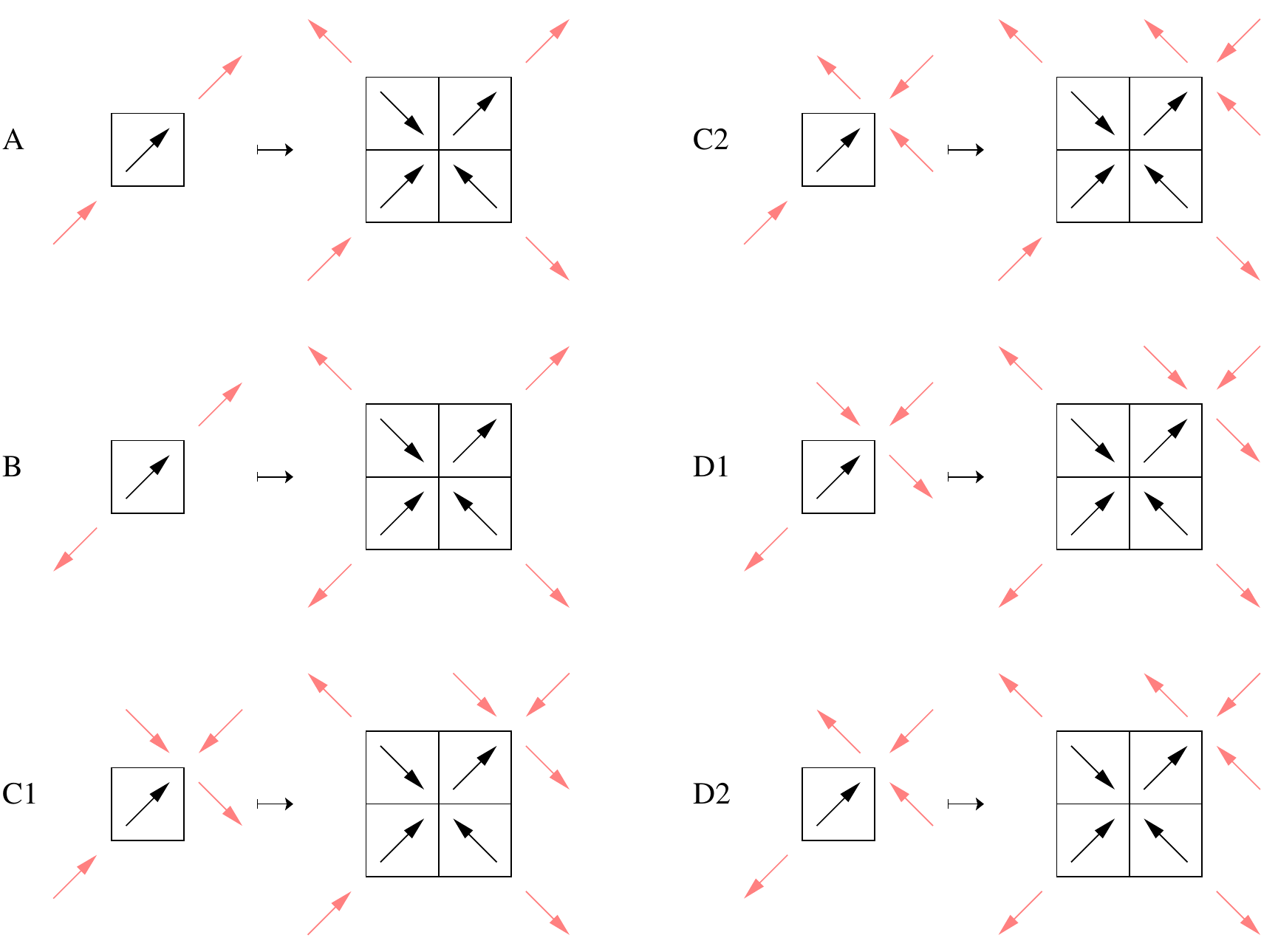}
\caption{{\small The 6 decorated tiles of the chair tiling (up to rotation), and how they substitute.}}
\label{fig-tiles}
\end{center}
\end{figure}

\begin{figure}[htp]
\begin{center}
\includegraphics[scale=0.65]{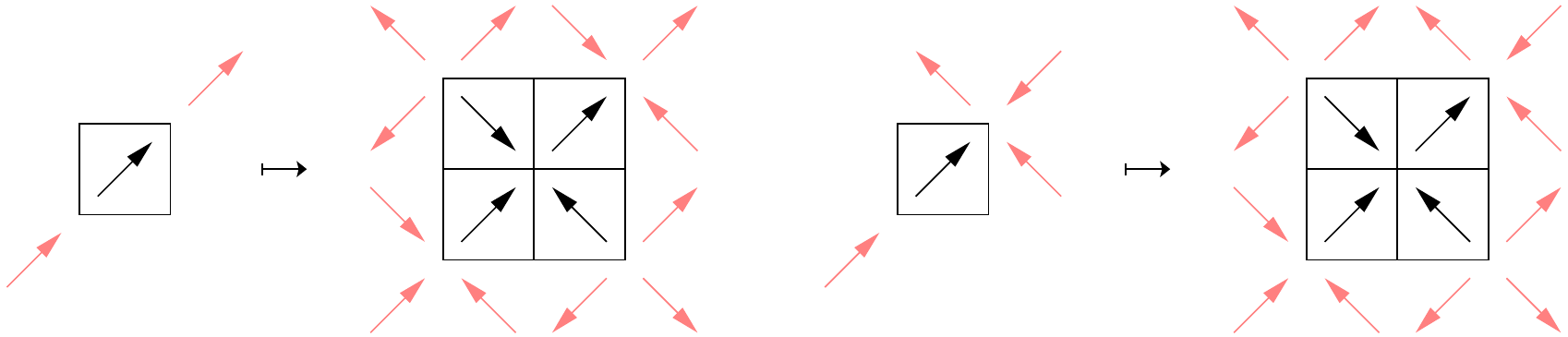}
\caption{{\small The decoration \emph{forces its border}: given a decorated tile $t$, there is no ambiguity on the type of (undecorated) tiles which surround $\omega(t)$.}}
\label{fig:border-forcing}
\end{center}
\end{figure}

It can be proved (see~\cite{Sad08book}) that this substitution gives the same space as the simpler chair substitution. More precisely, the map from the decorated tiling space to the non-decorated tiling space is just ``forgetting'' the additional labels.
The inverse map simply adds the information about the neighbouring tiles.

The essential feature of this decorated substitution is that it \emph{forces its border} in the sense that given a tiling $T \in \Omega$ and a (decorated) tile $t \in T$, all tiles adjacent to $\omega(t)$ in $\omega(T)$ are entirely determined by $t$ (and do not depend on the neighbouring tiles of $t$). See Figure~\ref{fig:border-forcing} for an example.

Next, we want to identify the edges of the tiles, and how they substitute.
For our purpose, an edge is a set of two tiles $e = \{t_1,t_2\}$ whose intersection has dimension one.
One can think of it as the segment $t_1 \cap t_2$ with an additional label consisting of the two tiles $t_1,t_2$. By abuse of terminology, we will also let $e$ denote $t_1 \cap t_2$ (or we will refer to this set as ``the support of $e$'' when more precision is needed).
The substitution on tiles induces a substitution on edges as follows.
Let $e=\{t_1, t_2\}$, and $P = \omega (\{t_1, t_2\})$. The substituted edge $\omega^{(1)} (e)$ is equal to $\{e'_1, e'_2\}$, where $e'_1 \cup e'_2$ is a sub-patch of $P$, and the support of the tiles $e'_i$ cover exactly the support of $\lambda e$ (here, $\lambda = 2$ is the inflation factor for this substitution).

\begin{figure}[htp]
\begin{center}
\includegraphics[scale=0.65]{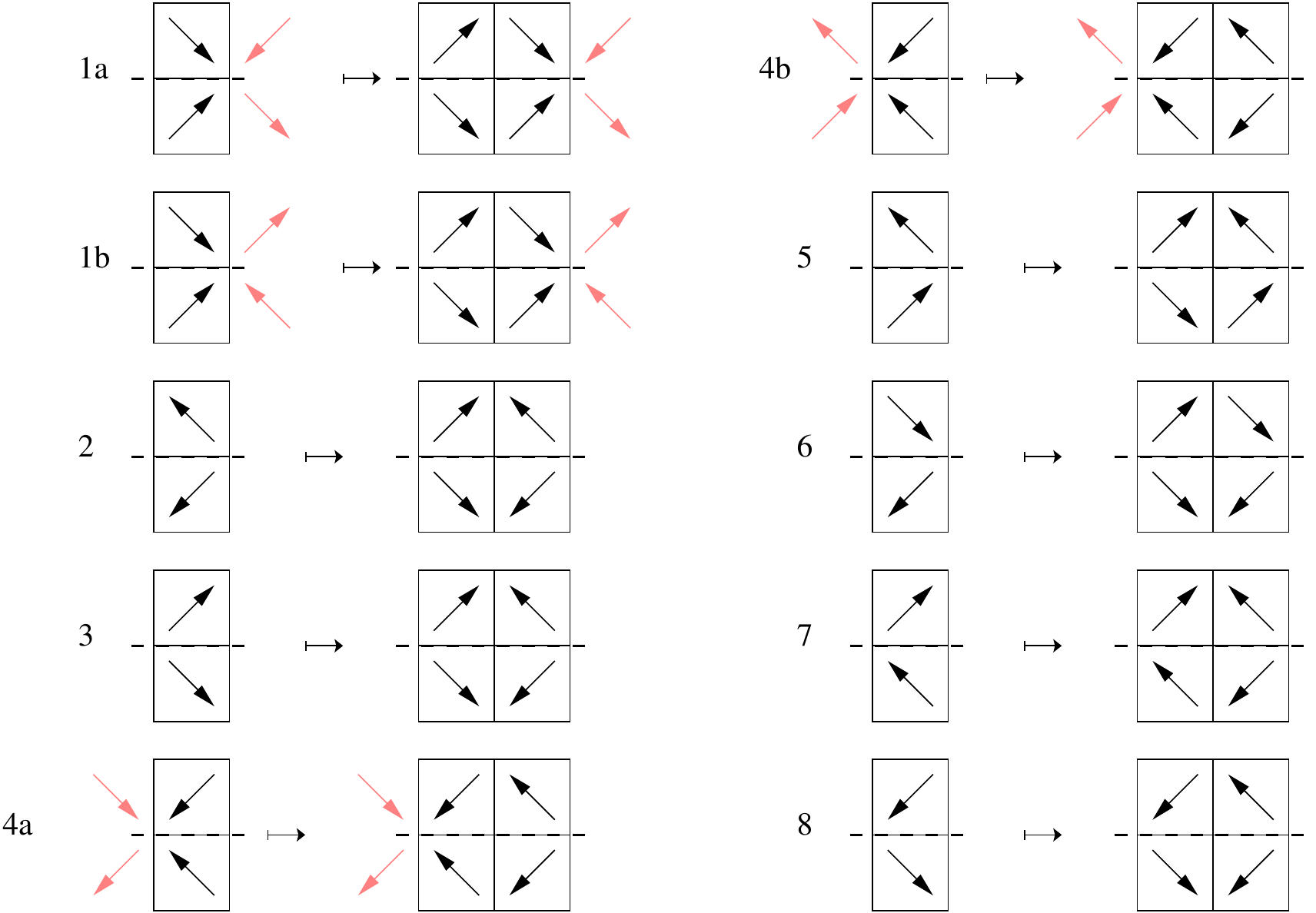}
\caption{{\small The 10 decorated horizontal edges of the chair tiling, and how they substitute. Note, for example, that the bottom tile of edge $1b$ (using additional information from the tile above and the two red tiles) could be of type either $A$ or $B$.}}
\label{fig-edges}
\end{center}
\end{figure}

We decorate edges in a similar way as the tiles, by adding information such that they satisfy the following border forcing condition: ``for all (decorated) edge $e \subset T$, the set of (undecorated) tiles which intersect the support of $\omega^{(1)}(e)$ is entirely determined by~$e$''.
This condition would be automatically satisfied if we chose edges to be pairs of decorated tiles.
However, since we want to keep the number of elements involved in the computations as low as possible, we first define edges to be pairs of undecorated tiles, and then add a decoration if needed.
As a consequence, a tile $t_1$ in an edge $\{t_1,t_2\}$ is less decorated than a decorated tile as defined above (``edges are less decorated than tiles'').
Figure~\ref{fig-edges} lists horizontal boundaries of the chair substitution, along with their images under $\omega^{(1)}$. Because of the symmetries, vertical boundaries are obtained by rotation.

Similarly, we want to identify the substitution induced on tile vertices. A vertex (or vertex star) for our purpose, is a set of four (undecorated) tiles intersecting at a point.
The induced substitution is defined as for edges, and automatically forces its border. For the purpose of $K$-theory computations, it is enough to consider vertex-stars which lie in the eventual range of the induces substitution. See Figure~\ref{fig-vertex-star}.

\begin{figure}[htp]
\begin{center}
\includegraphics[scale=0.65]{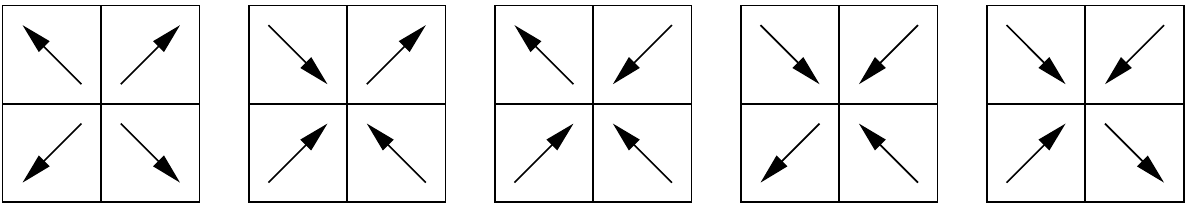}
\caption{{\small The vertex-stars of the substitution which lie in the eventual range. We label them respectively $v_0, \ldots v_5$. Any vertex star substitutes into one of these. They are fixed by the substitution.}}
\label{fig-vertex-star}
\end{center}
\end{figure}

The last step of our setup is to examine adjacency: an edge $e = \{t_1, t_2\}$ is adjacent to a decorated tile $t$ if there is a tiling $T \in \Omega$ such that $e \subset T$, $t-x \in T$ for some $x$, and the support of $t-x$ and $t_1$ (or $t_2)$ coincide.
For example, the edge of type $7$ as defined in Figure~\ref{fig-edges} can be glued at the bottom of tiles $A$, $C_1$ and $C_2$.
Similarly, we examine adjacency of edges and vertex stars.

\medskip

To summarize, we have:
\begin{itemize}
 \item A list of border-forcing substitutions in all dimensions, defined on decorated tiles, edges, vertex-stars. The Abelianization matrices are $B^{(k)}$ ($k=0, 1, 2$), and are defined by: $[B^{(k)}]_{i,j}$ is the number of occurences of tile (or edge, vertex) $i$ the substitution of tile (or edge, vertex) $j$.
 \item A list of adjacency relations between tiles, edges, vertices.
\end{itemize}

\subsection{Spaces and inverse limits}

Given an aperiodic substitution tiling space such as $\Omega$, it is known 
that the substitution $\omega$ acts on it by homeomorphism~\cite{Sol97}. In particular, it is invertible. As a consequence, for each tiling $T \in \Omega$ and each tile $t \in T$, there is a \emph{unique} tile $t' \in \omega^{-1}(T)$ such that $t \in \omega(t') \subset T$.
The patch $\omega(t')$ is called the \emph{super-tile} containing $t$.
This leads to the following definition:
\begin{Definition}\label{def:supertile}
 Given a tiling $T$, and $t \in T$, the $n$-th order supertile containing $t$ is the unique patch of the form $\omega^n(t')$ which contains $t$ and such that $t' \in \omega^{-n}(T)$. We may write ``a $n$-th order supertile of type $C_1$'', for example, when the tile $t'$ is a representative of the prototile $C_1$.
 An infinite-order supertile is an increasing union of $n$-th order supertiles as $n \rightarrow +\infty$.
\end{Definition}

\begin{Definition}\label{def:boundary}
 We say that $T$ is a \emph{boundary tiling} if $0$ belongs to two or more infinite order supertiles. In our case of tilings by squares, whenever $T$ is a boundary tiling, then either $T - (x,0)$ is a boundary tiling for all $x \in \R$, or $T-(0,x)$ is a boundary tiling, or both. In the first case, we say that $T$ has a horizontal boundary. In the second case we say that it has a vertical boundary.
\end{Definition}

It is easily checked that the set of boundary tilings is closed in $\Omega$.

Such aperiodic substitution spaces can be described as inverse limits~\cite{AP98}.
A CW-complex $\Gamma$ is defined as follows:
let $T$ be a tiling in $\Omega$. Instead of a set of tiles, we see it as a CW-decomposition of $\R^2$.
Let $\Gamma$ be the quotient of this CW-complex in which two $2$-cells (resp.\ $1$-cells or $0$-cells) are identified whenever they correspond to the same decorated tile (resp.\ edge or vertex-star).
A $2$-cell and a $1$-cell are adjacent in $\Gamma$ if and only if the corresponding tile and edge are adjacent in the sense of the previous section.

The substitution induces a continuous, cellular map on $\Gamma$.
Since the decoration was chosen in such a way that $\omega$ forces its border (see~\cite{AP98} and~\cite{Sad08book}) it results that
\[
 \Omega \simeq \varprojlim (\Gamma \leftarrow \Gamma \leftarrow \ldots),
\]
where $\simeq$ stands for ``is homeomorphic''.

Any tiling $T$ corresponds to a sequence of elements $(\gamma_1, \gamma_2, \ldots)$ in $\Gamma$. The map is defined in such a way that $x_i$ belongs to a tile of type $t$ in $\Gamma$ if and only if $0$ belongs to a tile of type $t$ in $\omega^{-i}(T)$, or equivalently if and only if $0$ belongs to a $n$-supertile of type $t$ in $T$.
It is easily seen that a tiling $T \in \Omega$ is:

\begin{itemize}
\item  a boundary tiling if and only if its representation in the inverse limit $(\gamma_1, \gamma_2, \ldots)$ satisfies that $\gamma_i$ belongs to the $1$-skeleton for all $i$; 
\item both a vertical and horizontal boundary tiling if and only if the $\gamma_i$'s belong to the $0$-skeleton.
\end{itemize}

The CW-complex $\Gamma$ can be filtered by its $0$- and $1$-skeletons. We define $\Omega^{(1)}:= \varprojlim (\Gamma^{(1)} \leftarrow \ldots)$ the inverse limit of the $1$-skeletons. Note that this $1$-skeleton consists of edges of prototiles.
Similarly, we define $\Omega^{(0)}$ as the inverse limit of the $0$-skeletons.
We can be a bit finer and define $\Omega^{(1)}_h$ and $\Omega^{(1)}_v$ the inverse limits of the part of the $1$-skeleton which correspond to horizontal (resp.\ vertical) edges of tiles. These two sets intersect exactly at $\Omega^{(0)}$ and their union is $\Omega^{(1)}$.
They are the sets of boundary tilings (resp.\ horizontal, vertical boundary tilings).

\begin{Definition}\label{def:transversal}
 The \emph{canonical transversal} of $\Omega$ is the set $\Xi$ of all tilings $T \in \Omega$ such that $0$ lies at the barycenter of a tile.
 Given the geometry of the tiles, the $\R^2$-action restricts to a $\Z^2$ action.
 We define similarly the set $\Xi^{(1)}$ as the set of all tilings in $\Omega^{(1)}$ which have the center of an edge at the origin.
 We define $\Xi^{(1)}_v$ and $\Xi^{(1)}_h$ as the analogous sets for $\Omega^{(1)}_v$ and $\Omega^{(1)}_h$. Notice how these last two sets carry a $\Z$-action by vertical (resp.\ horizontal) translation.
\end{Definition}

\section{Definitions of the groupoids}
\label{sec-def}

 We define here the groupoids associated with the tiling space.
 Given a tiling space $\Omega$, the action of $\R^2$ allows to define a continuous groupoid, noted $\Omega \rtimes \R^2$.
 It consists of pairs $(T,x)$ with partially defined product:
 \[
  (T,x) \cdot (T',y) = (T',x+y) \qquad \text{if } T'-y = T.
 \]
 It also comes with a range and source map respectively defined by $r(T,x) = T-x$ and $s(T,x) = T$. The topology is the one induced by $\Omega \times \R^2$.
 This groupoid however is not the most tractable one.
 Let $G$ be the groupoid of the transversal, which is the reduction of $\Omega \rtimes \R^2$ to the transversal $\Xi$: $G = s^{-1}(\Xi) \cap r^{-1} (\Xi) \subset \Omega \rtimes \R^2$.
 This groupoid is \'etale, meaning that $s$ and $r$ are local homeomorphisms. In particular, the orbits (set of the form $r^{-1} (T,x)$) are discrete, and the counting measure defines a Haar system. There is therefore a well-defined reduced $C^*$-algebra associated with them.
 In our case, $G$ identifies with the groupoid $\Xi \rtimes \Z^2$ given by the action of $\Z^2$ on $\Xi$.
 The subset $\Xi$ is an abstract transversal in the sense of Muhly--Renault--Williams~\cite{MRW87}, and therefore $\Omega \rtimes \R^d$ and $G$ are Morita-equivalent.
 
 Similarly, we define $G^{(1)}_h$ and $G^{(1)}_v$ respectively as the reductions of $\Omega \rtimes \R^2$ to $\Xi^{(1)}_h$ and $\Xi^{(1)}_v$.
 Note that these ``transversals'' are abstract transversals of $\Omega^{(1)}_h$ and $\Omega^{(1)}_v$ respectively, but are \emph{not} transversals of $\Omega$, as they don't intersect every orbit.
 Therfore, $G^{(1)}_h$ and $G^{(1)}_v$  are not Morita-equivalent to $\Omega \rtimes \R^2$ or $G$.
 Note however that they are \'etale (they can be described as being given by the $\Z$-action on $\Xi_v^{(1)}$ or $\Xi_h^{(1)}$ by vertical or horizontal translations respectively).
 As an irritating remark, the reduction of $\Omega \rtimes \R^2$ to $\Xi^{(1)}$ is unfortunately not \'etale.
Putnam's constructions, described in the next section, allow one to go around this difficulty.

\subsection{Putnam's disconnections and residual groupoids}
\label{ssec-decres}

We now describe how the tiling groupoid can be decomposed into a disconnected groupoid and a residual groupoid.
Consider a discrete tiling space $\Xi$ of a substitution tiling of the plane.
Given a subspace of boundary tilings (for example $\Omega^{(1)}_h$), define the subspace of $\Xi$ which consists of all tilings which contain a boundary (not necessarily at the origin):
\[
 \Xi_h := \Xi \cap \{T-x \ ; \ T \in \Omega^{(1)}_h, \ x \in \R^2 \}.
\]
As an algebraic object, let $G^\mathrm{res}_h$ be the reduction on $\Xi_h$ of $G$. We define a topology on this groupoid as follows:
given a sequence $(T_n,x_n)$ in $G^\mathrm{res}_h$, we say that it converges to $(T,x)$ if $T_n$ can be written as $T^{(1)}_n + y_n$ with $T^{(1)}_n \in \Omega^{(1)}_h$, $y_n \in \R^2$, such that $T^{(1)}_n \rightarrow T^{(1)} \in \Omega^{(1)}$, $y_n \rightarrow y$, $x_n \rightarrow x$ and $T^{(1)}-y = T$.
This topology turns it into an \'etale groupoid.

We write $\Xi_h = \Xi_h^+ \cup \Xi_h^-$ where $\Xi_h^+$ consists of all tilings of $\Xi_h$ of the form $T-(x_1,x_2)$ with $T \in \Omega^{(1)}_h$ and $x_2 > 0$ (tilings for which the origin is above the boundary). $\Xi_h^-$ is defined as the complement.
Similarly (and by convention), denote $\Xi_v^+$ the tilings on the left of the vertical boundary and $\Xi_v^-$ those on the right.

Let $G$ be the groupoid of $\Xi$.
We define the \emph{disconnection of $G$ along the horizontal boundary} to be the groupoid obtained from $G$ by removing all arrows wich cross the horizontal boundary:
\[
G^{\mathrm{dec}}_h = \bigl\{\gamma \in G \ ; \  s(\gamma) \in \Xi_h^{\pm} \Rightarrow r(\gamma) \in \Xi_h^\pm  \bigr\},
\]
with the topology induced from the topology of $G$.
It consists of all arrows of $G$ which don't cross a horizontal boundary line.
Clearly $G^{dec}_h$ and $G$ have the same orbits, except for the orbits of the boundary tilings.

\begin{Proposition}
\label{prop-morita}
The groupoid $G^{\mathrm{res}}_h$ is Morita equivalent to the groupoid of the $\Z$-dynamical system $G^{(1)}_h = \Xi^{(1)}_h \rtimes \Z$.
\end{Proposition}

\begin{proof}
 Define a groupoid consisting of all pairs $(T,x)$ where $T$ is in the $\R^2$-orbit of an element of $\Omega^{(1)}_h$ and $x \in \R^2$. Define convergence similarly as for $G^\mathrm{res}_h$: $(T_n,x_n)$ converges if $x_n$ converges and $T_n$ can be written as $T'_n-y_n$ with $y_n$ converging in $\R^2$ and $T'_n$ converging in $\Omega^{(1)}_h$.
 Then $\Xi_h$ and $\Xi^{(1)}_h$ are both transversals in the sense of Muhly Renault and Williams (in particular notice that they are closed for this topology). Therefore, $G^\mathrm{res}_h$ and $G^{(1)}_h$ are reductions of the same groupoid on two transversals---hence they are Morita equivalent.
\end{proof}

\medskip

The goal is to iterate several disconnections in a row.
First we disconnect $G$ along {\em horizontal} (resp.\@ \emph{vertical}) boundaries: this yields $G^{\mathrm{dec}}_h$ (resp. $G^{\mathrm{dec}}_v$).

Next we disconnect $G_h$ along {\em vertical} boundaries: this yields a groupoid which we call $G_{AF}^{(2)}$:
\(
(G^{dec}_h)^{dec}_v=(G^{dec}_v)^{dec}_h=G_{AF}^{(2)}.
\)
\begin{Proposition}
 The groupoid $G_{AF}^{(2)}$ obtained after two disconnections is (as its name indicates) an $AF$ groupoid.
\end{Proposition}

\begin{proof}
 An element $(T,x)$ belongs to $G_{AF}^{(2)}$ if and only if $T$ and $T-x$ are not separated by a vertical or horizontal infinite boundary. 
 Since there are no other boundaries than vertical or horizontal (given the geometry of the tiles), it means that in $T$, $0$ and $x$ are contained in the same super-tile of some order.
 Let us define the equivalence relation $R_n$ on $\Xi$ by: ``$T \sim_{R_n} T'$ if and only if $T'=T-x$ and the points $0$ and $x$ are in the same $n$-order supertile of $T$''.
 One proves that $G_{AF}^{(2)}$ is the direct limit of the equivalence relations $R_n$. It is easy to check that the $R_n$ are compact and \'etale (see~\cite[Definition~3.1]{GPS04}), so $G_{AF}^{(2)}$ is indeed an $AF$ equivalence relation.
\end{proof}

After a second disconnection, the residual groupoid is
\(
(G^{dec}_h)^{res}_v.
\)
It consists of elements of $G^\mathrm{dec}_h$ which ``stay close to a vertical boundary''.
Equivalently, it consists of elements of $G$ which both don't cross a horizontal boundary and stay close to a vertical boundary.
In other words, it consists of elements of $(G^\mathrm{res}_v)^\mathrm{dec}_h$.
The first part of this proposition is proved similarly to Proposition~\ref{prop-morita}.
\begin{Proposition}
 The groupoid $(G^{dec}_h)^{res}_v$ is Morita equivalent to $(G^{(1)}_v)^\mathrm{dec}_h$.
 We call $G^{(1)}_{AF,v} := (G^{(1)}_v)^\mathrm{dec}_h$. It is an $AF$ groupoid.
\end{Proposition}

\begin{proof}
For the first part (Morita equivalence of $(G^{dec}_h)^{res}_v$ and $(G^{(1)}_v)^\mathrm{dec}_h$), the proof is identical to Proposition~\ref{prop-morita}.

$G^{(1)}_v$ is the groupoid of the $\Z$-action on the Cantor set $\Xi^{(1)}_v$ by vertical translations.
Its subgroupoid $G^{(1)}_{AF,v}$ consists of elements are of the form $(T,x)$ where $T, T-x \in \Xi^{(1)}_v$ and such that $T$ and $T-x$ are not separated by a horizontal boundary.
If we see $\Xi^{(1)}_v$ as sitting in the inverse limit and represent $T$ and $T-x$ by $(\gamma_1, \gamma_2, \ldots)$ and $(\gamma'_1, \gamma'_2, \ldots)$ respectively, it then means that $\gamma_i$ and $\gamma'_i$ belong to the same vertical edge of $\Gamma$ for all $i$ big enough.
We can use this observation to write an increasing sequence of compact \'etale equivalence relations whose limit is $(G^{(1)}_v)^\mathrm{dec}_h =: G^{(1)}_{AF,v}$. So it is indeed an $AF$ groupoid.
\end{proof}

Let us summarize which groupoids appear above, as well as the equivalences between them.

\begin{enumerate}
\item The groupoids associated to $2$-dimensional systems, as well as their disconnected versions: $G$ is the groupoid of the $\Z^2$-action; $G^\mathrm{dec}_{h/v}$ is the groupoid obtained from $G$ by removing all arrows which pass through a horizontal (resp.\ vertical) boundary; $G^{(2)}_{AF}$ is the groupoid obtained from $G$ by removing arrows which pass through a horizontal \emph{or} a vertical boundary, which is an $AF$ groupoid.
\item The groupoids associated with $1$-dimensional systems:
  \begin{enumerate}
  \item $G^\mathrm{res}_h$ and $G^{(1)}_h$ (resp.\ $G^\mathrm{res}_v$ and $G^{(1)}_v$) are Morita equivalent. The first consists of all arrows which pass through a horizontal (resp.\ vertical) boundary, with \emph{an appropriate topology} which makes it \'etale (and makes its space of units complete).
  The second is the groupoid of the $\Z$-action on $\Xi^{(1)}_h$ (resp.\ $\Xi^{(1)}_v$).
  \item $G^{(1)}_{AF,v}$ and $G^{(2)}_{AF,h}$ are the $AF$-subgroupoids of respectively $G^{(1)}_v$ and $G^{(1)}_h$, which are obtained by disconnecting the orbits at points of $\Omega^{(0)}$.
  Because of the particular substitution we are using (with rotational symmetry of angle $\pi/4$), these two groupoids are isomorphic.
  \end{enumerate}
 \item The groupoid associated with a $0$-dimensional system: it is the trivial groupoid $\Omega^{(0)}$, with no arrows except the units.
\end{enumerate}

\section{Putnam's exact sequences for square tilings of the plane}
\label{sec-PutES}

The $K$-theory of a groupoid and its disconnected version are related by a $6$-term exact sequence, as established by Putnam.
The result first appeared for $\Z$-actions on a Cantor set~\cite{Put89}, and was later generalized for more general groupoids and disconnections.

Let us state first the statement for $\Z$-actions. It was originally stated for free, minimal $\Z$-actions, but in the light of the later papers~\cite{Put97, Put98}, this assumption is not essential if the statement is made in the following form.

\begin{Theorem}
 Let $G^{(1)}_h$ be the groupoid of the $\Z$-action on $\Xi^{(1)}_h$, and $G^{(1)}_{AF,h}$ its $AF$-subgroupoid obtained by disconnecting the orbits at $\Omega^{(0)}$. Then we have the following exact sequence:
 \[
  0 \rightarrow \Z \rightarrow C(\Omega^{(0)};\Z) \rightarrow K_0 \big( C^\ast (G^{(1)}_{AF,h}) \big)
    \rightarrow K_0 \big( C^\ast (G^{(1)}_h) \big) \rightarrow 0.
 \]
\end{Theorem}

 Note that the group $C(\Omega^{(0)};\Z)$ is $\Z^{\# \Omega^{(0)}}$. By theorem, the first and last nontrivial arrows are the map $1 \mapsto (1, \ldots, 1)$ and induced by inclusion of $C^\ast$-algebras, respectively.
 Let us describe the second arrow.
 Let $T_0$ be a point in $\Omega^{(0)}$. Then there is a smallest $x>0$ (equal to $1/2$ if the tiles have unit length) such that $T_0 + (x,0)$ and $T_0 - (x,0)$ belong to $\Xi_h^{(1)}$. Call $T_0^+$ and $T_0^-$ these two tilings respectively.
 In the inverse limit construction, $T_0^\pm = (\gamma^\pm_0, \gamma^\pm_1, \ldots)$.
 
 Let
 \[
 U^\pm_n = \{ (\gamma^\pm_0, \ldots, \gamma^\pm_n, *, *, \ldots) \} \subset \varprojlim_n \Gamma.
 \]
 Then for some $n$ big enough, the class in $K_0$ of $[\chi(U^+_n)] - [\chi(U^-_n)]$ does not depend on $n$ anymore.
 The image of the generator corresponding to $T_0$ in $X^{\# \Omega^{(0)}}$, is defined to be this $K$-theory class.

\begin{Theorem}
\label{thm-ESPut}
Let $\mathcal G$ be a groupoid, and $b \in \{h,v\}$ be a boundary. In the following $\mathcal G$ will be either $G$ or $G^\mathrm{dec}_h$.
Then, up to Morita equivalence, $\mathcal G^\mathrm{res}_b$ is respectively $G^{(1)}_b$ or $G^{(1)}_{AF,b}$. In both cases, its unit space is $\Xi ^{(1)}_b$.

Putnam's exact sequence for the disconnection of $\mathcal G$ along a boundary $b \in \{h,v\}$, as described in Section~\ref{ssec-decres} reads:
\begin{equation}
\label{eq-ES-Putnam}
\xymatrix{
K_0 \bigl( C^\ast(\mathcal G^{res,b}) \bigr) \ar@{->}[r]^{\beta_0} & 
K_0 \bigl( C^\ast(\mathcal G^{dec,b}) \bigr) \ar@{->}[r]^{\gamma_0} & 
K_0 \bigl( C^\ast(\mathcal G) \bigr) \ar@{->}[d]^{\alpha_1} \\
K_1 \bigl( C^\ast(\mathcal G) \bigr) \ar@{->}[u]^{\alpha_0} &  
K_1 \bigl(  C^\ast(\mathcal G^{dec,b})\bigr)  \ar@{->}[l]^{\gamma_1} & 
K_1 \bigl( C^\ast(\mathcal G^{res,b}) \bigr) \ar@{->}[l]^{\beta_1}
}
\end{equation}
The maps $\gamma$ are induced by the inclusion $\mathcal G^{dec,b} \subset \mathcal G$.
\end{Theorem}

Let us describe the maps $\alpha$ and $\beta$ in our setting.
A full description of these maps is left for a future paper. We just give them in our setting (two-dimensional tiling by squares).
The map $\beta_0$ is similar to the map described for the previous exact sequence: assume for example that $b=h$ so that the disconnection is done along the horizontal boundary.
The $C^*$-algebra $C^\ast(\mathcal G^{res,h})$ is Morita equivalent to either $C^\ast( G^{(1)}_h)$ or $C^\ast(G^{(1)}_{h,AF})$. In either case, these algebras are generated by elements of the form $\ind(U) u^n$, where $\ind(U)$ is the indicator function of the clopen set $U$ of $\Xi^{(1)}_h$, and $n \in \mathbb N$, and their $K$-theory is described by the classes of projections of the form $\ind(U)$, for $U$ a clopen set.

We may assume that $U$ has the form (as seen in the inverse limit description):
\[
 U = U(x_1, x_2, \ldots, x_k) = \{ (x_1, x_2, \ldots, x_k, \gamma_{k+1}, \ldots) \ ; \ \gamma_i \in \Gamma^{(1)} \}.
\]

Here, the $x_i$'s and $\gamma_i$'s are elements in $\Gamma^{(1)}$ (the $x_i$'s are fixed and define~$U$).
Let
\[
V(n) = V(x_1, \ldots, x_k; n) := \{ (x_1, x_2, \ldots, x_k, \gamma_{k+1}, \ldots, \gamma_n, \gamma'_{n+1} \ldots) \ ; \ \gamma_j \in \Gamma^{(1)}, \ \gamma'_i \in \Gamma \}.
\]

It is a subset of $\Omega$ whose intersection with $\Omega^{(1)}_h$ is $U$. We call it a {\em thickening of order $n$} of $U$.
Now, let $V(n)^\pm := V(n) - (0,\pm 1/2)$. Each of these sets is an open set in $\Xi$.
Finally, for a finite sum $\sum_k c_k \ind(U_k) u^k$ and $n \in \N$, define the pair of maps
\[
 \phi_n^\pm: \sum_k c_k \ind(U_k) u^k \mapsto \sum_k c_k \ind(V_k(n)^\pm) u_{(1,0)}^k, 
\]
where $u_{(1,0)}$ is the unitary in $C(\Xi) \rtimes \Z^2$ which implements the translation by $(1,0)$.
It induces a map in $K$-theory for $n$ big enough, and the difference $(\phi^+_n)_* - (\phi^-_n)_*$ applied to a $K$-theory element doesn't depend on $n$ if $n$ is big enough.

For our purpose, since $K_0(C^\ast(\mathcal G^{res,h}))$ is generated by classes of characteristic functions of clopen sets, we just need:
\[
\beta_0 ([\ind(U)]) := [\ind(V(n)^+] - [\ind(V(n)^-)],
\]
and
\[
\beta_1 ([u]) := [u_{(1,0)}] - [u_{(1,0)}^*] = 0.
\]

The map $\alpha_0$ is given by a $K\! K$-theory class (see~\cite{Put98}), and is formally defined as follows.
Let $\chi^+$ be the indicator function of $\Xi^+_h$.
Let $E:= C^\ast(\mathcal G^{res,h})$ be viewed as a Hilbert right $C^*$-module over itself.
It carries a left $C^\ast(\mathcal G)$-action, which makes it a bi-module. Note also that $\chi^+$ and $(1-\chi^+)$ also act on the left.
For $[u] \in K_1 \bigl( C^\ast(\mathcal G) \bigr)$, let $U:= \chi^+ u \chi^+ + (1-\chi^+)$.
It defines a linear map of the module $E$ by multiplication on the left. We then define
\[
 \alpha_0 ([u]) := [\Ker (U)] - [\Ker (U^*)].
\]
Here, the $K$-theory class is defined in terms of formal difference of (classes of) finitely generated projective modules over the algebra. 
Remark that the modules $\Ker (U)$ and $\Ker (U^*)$ are submodules of $\chi^+ E$, so it is equivalent to define $\alpha_0$ by $\alpha_0([u]) = [\Ker (\chi^+ u \chi^+)] - [\Ker (\chi^+ u^* \chi^+)]$, where $\chi^+ u \chi^+$ is viewed as a map $\chi^+ E \rightarrow \chi^+ E$.
In our example, these modules are complementable, and define $K$-elements.

\begin{figure}[htp]
\begin{center}
\includegraphics[scale=0.8]{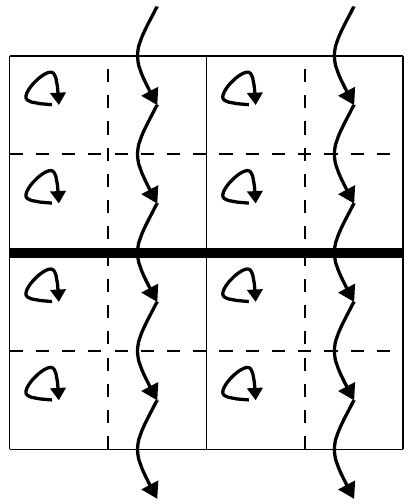}
\includegraphics[scale=0.8]{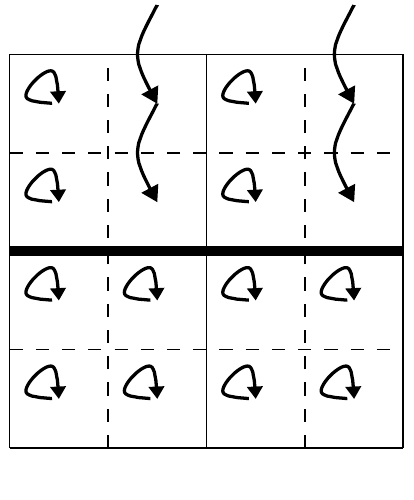}
\includegraphics[scale=0.8]{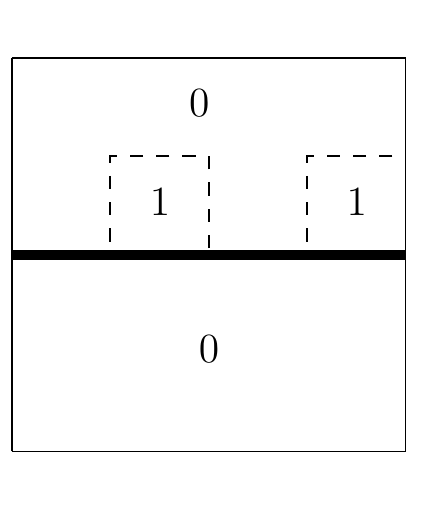}
\caption{{\small On the left, the unitary $(u^{(1)})^*$: the unitary, seen as a function on $C_c(G)$, is $1$ on the arrows pictured, and $0$ elsewhere. Alternatively, seen as acting on the unit space $\Xi$, $(u^{(1)})^*$ is a translation by $(0,-1)$ on the tilings which have their origin in the right half of a $1$-supertile, and acts trivially on other tilings.
Middle: the operator $\chi^+ (u^{(1)})^* \chi^+ + (1-\chi^+)$. Right, the projection on the kernel of this operator.}}
\label{fig:unitaries}
\end{center}
\end{figure}

\begin{Example} \label{ex-alpha}
 Let us compute the image of a family of unitaries in $C^\ast (G)$ under $\alpha_0$.
 Let $u^{(0)}:= u_{(0,1)}$ be the unitary implementing vertical translation on $\Xi$ by vector $(0,1)$. Let $\chi_1$ be the indicator function of those tilings for which the tile at the origin in on the right-hand side of the supertile in which it sits.
 Define $\chi_n (T) := \prod_{i=1}^n \chi_1 (\omega^{-n+1} (T))$ (so that $\chi_n(T) = 1$ if and only if the tile at the origin of $T$ lies on the right-most side of the $n$-th order supertile containing the origin).
  Note that $\chi_1$ commutes with $u^{(0)}$.
  Define then $u^{(n)} := u^{(0)} \chi_n + (1-\chi_n)$. It is a unitary operator.
  See Figure~\ref{fig:unitaries}.
  
  It can be checked by hand that $[\Ker (\chi^+ u^{(1)} \chi^+ + (1-\chi^+))]$ is trivial, and $[\Ker (\chi^+ {u^{(1)}}^* \chi^+ + (1-\chi^+))]$ is generated by the indicator function $p$ defined by $p (T) := 1$ if and only if the tile at the origin is on the bottom of an infinite-order supertile (i.e.\@ is adjacent to an infinite boundary, above the boundary).
  Furthermore, $[\Ker (\chi^+ {u^{(1)}}^* \chi^+ + (1-\chi^+))]$ is generated by the projection $q(T) = p(T) \chi_n(T)$.
\end{Example}

\section{$K$-theory of the chair tiling}
\label{sec-Kth}

In this section, we compute the $K$-theory groups of the C$^\ast$-algebra of $G$, denoted \(K_\ast \bigl( C^\ast(G)\bigr)\).
It is done by a careful analysis of the substitution on tiles, edges and vertices, and of the adjacency relations.
What we have so far are the decorated tiles, edges and vertices, as well as the substitution matrices for each of these.
In the case of vertex-stars, it is enough to consider the substitution restricted to its eventual range, which is why Figure~\ref{fig-vertex-star} only pictures five vertex stars.
Iterations of the substitution on these produces fixed points.

\subsection*{$K_0$-groups of the $AF$-groupoids}
\label{ssec-dimgp}

We compute here the $K_0$-groups of the $AF$-groupoids $G_{AF}^{(2)}$, $G_{AFh}^{(1)}$, $G_{AFv}^{(1)}$, and $G_{AF}^{(0)}$, and describe their generators explicitly.

\subsubsection*{Method} This is the easy step: once the setup step has been done, we have substitution matrices written down. The $K$-groups are direct limits free abelian groups under maps given by these matrices. Using a computer program to compute the eigenvalues and eigenvectors of these matrices allows to compute the $K$-groups and find generators.

\subsubsection*{$K$-theory of $C^\ast ( G^{(2)}_{AF} )$}
The groupoid $G^{(2)}_{AF}=G_{AF}$ is the $AF$-groupoid of the tiling and its $K_0$-group is the dimension group of the substitution.
We consider prototiles modulo rotation $r=e^{\imath \pi k/2}$, $k=0,1,2,3$.
There are 6 decorated tiles up to rotation, which are shown in Figure~\ref{fig-tiles}.

We set $A^0,B^0,\ldots, D_2^0$ to be the prototiles with their outwards arrows pointing towards the north-east direction as shown in figure~\ref{fig-tiles}.
The substitution matrix can be written symbolically (using rotations):
\[
\left( 
\begin{array}{ccccccccccc}
2 & 1 & 1 & 1 & 0 & 0  \\
0 & 1 & 0 & 0 & 1 & 1  \\
0 & 0 & 1 & 0 & 1 & 0  \\
0 & 0 & 0 & 1 & 0 & 1  \\
r  & r & r & r & r & r  \\
r^3 & r^3 & r^3 & r^3 & r^3 & r^3  \\
\end{array}
\right)
\]

It acts on each summand of $\Z^{24}=\Z^6\oplus \ldots \oplus\Z^6$, with basis elements $A^0,B^0,\ldots, D_2^0$, $rA^0,rB^0,\ldots, rD_2^0$, $\ldots$, $r^3A^0,r^3B^0,\ldots, r^3D_2^0$.
It is possible to find the $K$-theory of this $AF$-algebra by looking at various representations of the rotation group $\Z / 4\Z$.
It is also possible to use mathematical software to write the $24 \times 24$ matrix and compute its eigenvalues and eigenvectors.
We let $B = B^{(2)}$ be the $24 \times 24$ substitution matrix.
The $K$-theory of the $AF$-algebra is
\[
 K_0\bigl( C^\ast(G_{AF}^{(2)}) \bigr) = \varinjlim \bigl( \Z^{24} \overset{B^T}{\longrightarrow} \Z^{24} \overset{B^T}{\longrightarrow} \ldots \bigr).
\]

We find here:

\begin{equation}
\label{K-Gaf}
K_0\bigl( C^\ast(G_{AF}^{(2)}) \bigr) \simeq \Z\bigl[\frac{1}{4}\bigr] \oplus \Z\bigl[\frac{1}{2}\bigr]^2 \oplus \Z^{12}, \qquad K_1\bigl( C^\ast(G_{AF}^{(2)}) \bigr) \simeq  0.
\end{equation}

Generators can be chosen as follows:
\begin{itemize}
\item $\Z\bigl[\frac{1}{4}\bigr]$: \( \ind_A+ \ind_B + \ind_{C_1} + \ind_{C_2} + \ind_{D_1}+  \ind_{D_2}=(1+r+r^2+r^3) (\ind_{A^0}+ \ind_{B^0} + \ind_{C_1^0} + \ind_{C_2^0} + \ind_{D_1^0}+  \ind_{D_2^0})\), which corresponds to the constant function equal to $1$ on $\Xi$;
\item $\Z\bigl[\frac{1}{2}\bigr]^2$: \( (r^2-1) \bigl(\ind_{A^0} + \ind_{B^0} + \ind_{C_1^0} + \ind_{C_2^0} + \ind_{D_1^0} + \ind_{D_2^0} \bigr)\) and
\( r(r^2-1) \bigl(\ind_{A^0} + \ind_{B^0} + \ind_{C_1^0} + \ind_{C_2^0} + \ind_{D_1^0} + \ind_{D_2^0} \bigr)\).
These corresponds to projections with constant value $\pm 1$ on tiles with ``oposite orientations'' (e.g.\@ $-1$ on tiles with a north-east arrow and $+1$ on tiles with a south-west arrow);
\item  $\Z^{12}$:
   \(\ind_{A}+\ind_{C_1} - \ind_{D_2}\),
   \(\ind_A+\ind_B-\ind_{C_1}-\ind_{D_1}\),
   \(\ind_B-\ind_{C_1}-\ind_{C_2}\) (rotational invariance of order $4$), 
   \((1-r+r^2-r^3)\bigl( \ind_{A^0}+\ind_{B^0}+ \ind_{C_2^0}+\ind_{D_2^0}\bigr)\),
   \((1-r+r^2-r^3)\bigl( \ind_{A^0}+\ind_{B^0}+ \ind_{C_1^0}+\ind_{D_1^0}\bigr)\),
   \((1-r+r^2-r^3)\bigl( -\ind_{B^0}+ \ind_{C_1^0}+\ind_{C_2^0}\bigr)\) (rotational invariance of order $2$),
   \( (1-r^2) \bigl(\ind_{B^0}- \ind_{C_1^0}+\ind_{D_2^0} - r(\ind_{A^0} -\ind_{B^0}) \bigr)\), and $r$ times this latter vector,
   \( (1-r^2) \bigl( \ind_{C_1^0}+\ind_{D_1^0} + r(\ind_{A^0}+\ind_{B^0})\bigr)\), and $r$ times this latter vector,
   \( (1-r^2) \bigl( -\ind_{B^0}+ \ind_{C_1^0}+\ind_{C_2^0}\bigr)\), and $r$ times this latter vector.
\end{itemize}
For later purpose we notice that the following vectors are in the kernel of the transpose of the substitution matrix (which is of dimension 8):
\(\ind_A+ \ind_B + \ind_{C_1} + \ind_{C_2} - 2\ind_{D_1}\),
\(\ind_A+ \ind_B + \ind_{C_1} + \ind_{C_2} - 2\ind_{D_2}\) (rotational invariance of order $4$), 
\((1-r+r^2-r^3)\bigl( \ind_{A^0}+\ind_{B^0}+ \ind_{C_1^0}+\ind_{C_2^0}+2\ind_{D_1^0}\bigr)\),
\((1-r+r^2-r^3)\bigl( \ind_{A^0}+\ind_{B^0}+ \ind_{C_1^0}+\ind_{C_2^0}+2\ind_{D_2^0}\bigr)\) (rotational invariance of order $2$),
\( (1-r^2) \bigl( 2\ind_{D_2^0} - r(\ind_{A^0}+\ind_{B^0}+\ind_{C_1^0}+\ind_{C_2^0})\bigr)\), and $r$ times this latter vector, 
\( (1-r^2) \bigl( 2\ind_{D_1^0} + r(\ind_{A^0}+\ind_{B^0}+\ind_{C_1^0}+\ind_{C_2^0})\bigr)\), and $r$ times this latter vector.

\subsubsection*{$K$-theory of $C^\ast ( G^{(1)}_{AFh} )$ and $C^\ast ( G^{(1)}_{AFv} )$}
The groupoid $G^{(1)}_{AFh}$ is the $AF$-groupoid of the $1$-dimensional substitution induced on horizontal edges.
Its $K_0$-group is the dimension group of this induced $1$-dimensional substitution.

There are 10 decorated edges as shown in Figure~\ref{fig-edges}.
We consider proto-edges modulo rotations $\rho=e^{\imath k \pi}$, $k=0,1$.
We set
\( a=[1_a], b=[1_b], c=[2], d=[5], e=[6]\), and specifiy edges with a given orientation: \(a^0=1_a, b^0=1_b, c^0=2, d^0=5, e^0=6\) with the orientations as shown in Figure~\ref{fig-edges}.

The substitution matrix reads:

\begin{equation}
\label{subst-edges}
\left( 
\begin{array}{ccccc}
1 & 0 & 0 & 0 & 0   \\
0 & 1 & 0 & 0 & 0   \\
\rho & \rho & 1+\rho & \rho & \rho  \\
0 & 0 & 0 & 1 & 0 \\
0 & 0 & 0 & 0 & 1 
\end{array}
\right)
\end{equation}

It acts on each summand of $\Z^{10}=\Z^5\oplus \Z^5$, writting the basis elements $a^0,\ldots e^0$ and $\rho a^0, \ldots, \rho e^0$.
It also corresponds to a $10 \times 10$ substitution matrix $B^{(1)}$, and the $K$-theory of $G^{(1)}_{AFh}$ is the direct limit of $\Z^{10}$ under the transpose of this matrix.
We find:

\begin{equation}
\label{K-G1AFh}
K_0\bigl( C^\ast(G^{(1)}_{AFh}) \bigr) \simeq \Z\bigl[\frac{1}{2}\bigr] \oplus \Z^8, \qquad K_1\bigl( C^\ast(G^{(1)}_{AFh}) \bigr) \simeq  0.
\end{equation}

Generators can be chosen as follows:
\begin{itemize}
\item $\Z[1/2]$ : \(\ind_{a} + \ind_{b} + \ind_{c} + \ind_{d} +\ind_e = (1+\rho) \bigl( \ind_{a^0} + \ind_{b^0} + \ind_{c^0} + \ind_{d^0} +\ind_{e^0}\bigr)\), which is the constant function equal to $1$ on $\Xi_h^{(1)}$.
\item $\Z^8$ : \(\ind_{a^0}, \ind_{b^0}, \ind_{d^0}, \ind_{e^0}\) and  \( \rho \ind_{a^0}, \rho \ind_{b^0}, \rho \ind_{d^0}, \rho \ind_{e^0}\).
\end{itemize}
Notice for further purpose that \( (1-\rho) \bigl( \ind_{a^0} + \ind_{b^0} + \ind_{c^0} + \ind_{d^0} +\ind_{e^0}\bigr)\) lies in the kernel of the transpose of the matrix. 

\medskip
The groupoid $G^{(1)}_{AFv}$ is the $AF$-groupoid of the $1$-dimensional substitution induced on vertical edges.
Its $K_0$-group is the dimension group of this induced $1$-dimensional substitution.
These two groupoids are isomorphic, via rotation by $\pi/2$.
Hence we have the same $K$-theory groups as for $G^{(1)}_{AFh}$ in equation~\eqref{K-G1AFh}:

\begin{equation}
\label{K-G1AFv}
K_0\bigl( C^\ast(G^{(1)}_{AFv}) \bigr) \simeq \Z\bigl[\frac{1}{2}\bigr] \oplus \Z^8, \qquad K_1\bigl( C^\ast(G^{(1)}_{AFv}) \bigr) \simeq  0.
\end{equation}

\subsubsection*{$K$-theory of $C^\ast ( G^{(0)}_{AF} )$}
As there are five fixed points of the substitution, we have 

\begin{equation}
\label{K-G0AF}
K_0\bigl( C^\ast(G^{(0)}_{AF} ) \bigr) \simeq \Z^5,\qquad K_1\bigl( C^\ast(G^{(0)}_{AF} ) \bigr) \simeq  0.
\end{equation}

Now the $\Z^5$ summand is generated by
\[
\Z^5 = \langle \ind_0, \ind_1, \cdots \ind_4 \rangle \simeq \langle \ind_0 + \ind_1 \cdots + \ind_4, \ind_1,\ind_2,\ind_3,\ind_4\rangle
\]
where $\ind_i$ is the constant projection on the vertex-star $v_i$.

\subsection{The maps between the $K$-groups}
\label{ssec-maps}

\subsubsection*{Method} Computing the map $\beta_0$ is done by using the formula after Theorem~\ref{thm-ESPut}, which in turn is done by careful analysis of the adjacency relation between tiles and edges, or vertices and edges.
Eventually, we will need to compute the cokernel of $\beta_0$, so we are interested in two pieces of information: the rank of the image of $\beta_0$, and whether the quotient has torsion. The first point is easy, as the generators that we get for the image of $\beta_0$ are essentially vectors in $\Z^{24}$ (or $\Z^{10}$), and the rank can be obtained by row reduction of the appropriate matrices.
Whether the quotient has torsion is done by computing the Smith normal form of an appropriate matrix.

\subsubsection*{Image of \(\beta_0^{(1)}: K_0\bigl( C^\ast(G^{(1)}_{AFv}) \bigr) \rightarrow K_0\bigl( C^\ast(G^{(2)}_{AF}) \bigr)\).} 
The edges pictured in Figure~\ref{fig-edges} are horizontal. We let $r \ind_{a^0}$, etc.\@ refer to the indicator functions of the vertical edges.
We have

\begin{itemize}
    \item \(\beta_0^{(1)} ( r \ind_{a^0} ) = (\ind_{{A^0}} + \ind_{B^0}) - r (\ind_{{C_1^0}} + \ind_{D_1^0}) \),
    \item \(\beta_0^{(1)} ( r \ind_{b^0} ) = (\ind_{{C_2^0}} + \ind_{D_2^0}) - r (\ind_{{A^0}} + \ind_{B^0}) \),
    \item \(\beta_0^{(1)} ( r \ind_{c^0} ) = (r^2 - r^3) (\ind_{{B^0}} + \ind_{D_1^0} + \ind_{D_2^0}) \),
    \item \(\beta_0^{(1)} ( r \ind_{d^0} ) = r^2 (\ind_{{A^0}} + \ind_{C_1^0} + \ind_{C_2^0}) - r (\ind_{{C_2^0}} + \ind_{D_2^0}) \),
    \item \(\beta_0^{(1)} ( r \ind_{e^0} ) = (\ind_{{C_1^0}} + \ind_{D_1^0}) - r^3 (\ind_{{A^0}} + \ind_{C^0_1} + \ind_{C^0_2}) \).
\end{itemize}

All of these elements belong to the eigenspace of the transpose of the substitution matrix associated with eigenvalue~$1$, except $\beta_0^{(1)}(r \ind_{c^0})$ which has a component in this eigenspace and a component in the kernel. A computation shows that $\beta_0^{(1)} (r \ind_{c^0})$ is equivalent, in $K$-theory to $(r-1)(\ind_{A^0} + \ind_{B^0} + \ind_{C_1^0} + \ind_{C_2^0} + \ind_{D_1^0} + \ind_{D_2^0})  + r^2(r-1) (\ind_{A^0} + \ind_{C_1^0} + \ind_{C_2^0})$, which is in the eigenspace associated with eigenvalue~$1$.

Notice that we have 
\[
\beta_0^{(1)} \rho =-r^2 \beta_0^{(1)}
\]
(indeed, $\rho$ is a rotation by angle $\pi$, so it acts like $r^2$; however, it also reverses top and bottom, hence the minus sign),
so we get for instance
\( \beta_0^{(1)} (r \ind_a) = \beta_0^{(1)} ((1+\rho) r \ind_{a^0})
                            = (1-r^2) \beta_0^{(1)} (r \ind_{a^0})
                            = (1-r^2)(\ind_{A^0} + \ind_{B^0}) - (r-r^3) (\ind_{C_1^0} + \ind_{D_1^0}) \).
We check by hand that $\beta_0^{(1)} (r(\ind_{a^0} + \ind_{b^0} + \ind_{c^0} + \ind_{d^0})) = 0$. In particular, the image under $\beta_0^1$ of the constant function equal to $1$ is $0$.

This linear relation (as well as the one obtained by multiplying by $\rho$) are the only ones. The range of $\beta_0^{(1)}$ is spanned (over $\Z$) by the images of $r \ind_{a^0}$, $r \ind_{b^0}$, $r \ind_{d^0}$, $r \ind_{e^0}$, $r \ind_{a^1}$, $r \ind_{b^1}$, $r \ind_{d^1}$, $r \ind_{e^1}$.
Therefore the image of $\beta_0^{(1)}$ is a subgroup of rank $8$ inside the $\Z^{12}$-summand of $K_0\bigl( C^\ast(G^{(2)}_{AF}) \bigr)$.
Assisted by a computer, we check that there exists a $\Z$-basis of $K_0\bigl( C^\ast(G^{(2)}_{AF}) \bigr)$ which contains $8$ elements in the range of $\beta_0^{(1)}$. In particular, the quotient is torsion-free.

\subsubsection*{Image of \(\beta_0^{(0)}: K_0\bigl( C^\ast(G^{(0)}_{AF}) \bigr) \rightarrow K_0\bigl( C^\ast(G^{(1)}_{AFh}) \bigr)\).} 
We have

\begin{itemize}
    \item \(\beta_0^{(0)}(\ind_0) = (\rho -1) \ind_{c^0}\), 
    \item \(\beta_0^{(0)}(\ind_1) = -\rho \beta_0^{(0)}(\ind_3) = \rho  \ind_{e^0} - \ind_{b^0}\), 
    \item \(\beta_0^{(0)}(\ind_2) = -\rho \beta_0^{(0)}(\ind_4) = \rho  \ind_{a^0} - \ind_{d^0}\).
\end{itemize}    
So we see that $\ind_0 + \ldots \ind_4$ is mapped to \((1-\rho)(\ind_{a^0}+ \ldots \ind_{e^0})\) which lies in the kernel of the transpose of the $1d$-substitution matrix.
Now the $\Z^8$ summand of \(K_0\bigl( C^\ast(G^{(1)}_{AFh})\) is generated by \(a^0,b^0,d^0,e^0,\rho a^0,\rho b^0,\rho d^0,\rho e^0\), or equivalently by \(a^0,b^0,d^0,e^0,\rho a^0 -d^0,\rho b^0-e^0,\rho d^0-a^0,\rho e^0-b^0\) and the last four vectors generate the image of $\beta_0^{(0)}$.
Hence, the image of the $\Z^8$ summand of $K_0\bigl( C^\ast(G^{(1)}_{AFh}) \bigr)$ in $\coker \beta_0^{(0)}$ is generated by \(a^0,b^0,d^0,e^0\).

\subsubsection*{Image of $\beta_0^{(2)}: K_0 \bigl( C^\ast (G^{(1)}_h) \bigr) \rightarrow K_0 \bigl( C^\ast (G_h^{dec}) \bigr)$.}
The computations are similar as for the image of $\beta_0^{(1)}$.
The images will naturally be elements of $K_0 \bigl( C^\ast (G^{(2)}_{AF}) \bigr)$, which map in $K_0 \bigl( C^\ast (G_h^{dec}) \bigr)$ (induced by the inclusion of the $C^\ast$-algebras).
The left-hand group is generated by \(a^0,b^0,d^0,e^0\).
We compute

\begin{itemize}
 \item $\beta_0^{(2)}(\ind_{a^0}) = - ( \ind_{C_1^0} + \ind_{D_1^0} ) + r^3( \ind_{A^0} + \ind_{B^0} )$,
 \item $\beta_0^{(2)}(\ind_{b^0}) = - ( \ind_{A^0} + \ind_{B^0} ) + r^3( \ind_{C_2^0} + \ind_{D_2^0} )$,
 \item $\beta_0^{(2)}(\ind_{d^0}) = - ( \ind_{C_2^0} + \ind_{D_2^0} ) + r( \ind_{A^0} \ind_{C_1^0} + \ind_{C_2^0} )$,
 \item $\beta_0^{(2)}(\ind_{e^0}) = - r^2 ( \ind_{A^0} + \ind_{C_1^0} + \ind_{C_2^0} ) + r^3 ( \ind_{C_1^0} + \ind_{D_1^0} )$.
\end{itemize}
We check that the images of these four elements in $K_0 (C^\ast (G^{(2)}_{AF}))$ lie in the eigenspace associated with eigenvalue $1$ of $B^T$.
They are also independent, and independent from the elements in the range of $\beta_0^{(1)}$.
Therefore, they are independent in $K_0 (C^\ast (G_h))$. Another computer-assisted check shows that the quotient is torsion-free.

\subsection{Computations of the $K$-groups}
\label{ssec-comp}

Let us now apply Putnam's exact sequence to the three disconnections decribed in Section~\ref{ssec-decres}, and use the three propositions we proved there.
Remember that the $K_1$-groups of $AF$-groupoids are trivial.
Putnam's exact sequence for the disconnection of $G$ along the horizontal boundary reads:

\begin{equation}
\label{eq-ES-d2}
\xymatrix{ 
K_0 \bigl( C^\ast(G_h^{res}) \bigr) \ar@{->}[r]^{\beta_0^{(2)}} & 
K_0 \bigl( C^\ast(G_h^{dec}) \bigr) \ar@{->}[r]^{\gamma_0^{(2)}} & 
K_0 \bigl( C^\ast(G) \bigr) \ar@{->}[d]^{\alpha_1^{(2)}} \\
K_1 \bigl( C^\ast(G) \bigr) \ar@{->}[u]^{\alpha_0^{(2)}} &  
K_1 \bigl(  C^\ast(G_h^{dec})\bigr)  \ar@{->}[l]^{\gamma_1^{(2)}} & 
K_1 \bigl( C^\ast(G_h^{res}) \bigr) \ar@{->}[l]^{\beta_1^{(2)}}
}
\end{equation}

Putnam's exact sequence for the disconnection of $G^{dec}_h$ along the vertical boundary reads (omitting the trivial groups):
\begin{equation}\label{eq-ES-d1}
\xymatrix{ 
 K_1 \bigl( C^\ast(G_h^{dec}) \bigr) \ar@{^{(}->}[r]^{\alpha_0^{(1)}} &
       K_0 \bigl( C^\ast(G^{(1)}_{AFv}) \bigr) \ar@{->}[r]^{\beta_0^{(1)}} &
       K_0 \bigl( C^\ast(G_{AF}^{(2)}) \bigr) \ar@{->>}[r]^{\gamma_0^{(1)}} &
       K_0 \bigl( C^\ast(G_h^{dec}) \bigr).
}
\end{equation}

And for the disconnection of $G^{res}_h$ along the vertical boundary, it reads:
\begin{equation}
\label{eq-ES-d0}
\xymatrix{ 
  K_1 \bigl( C^\ast(G^{res}_h) \bigr) \ar@{^{(}->}[r]^{\alpha_0^{(0)}}  &
  K_0 \bigl( C^\ast(G^{(0)}_{AF}) \bigr) \ar@{->}[r]^{\beta_0^{(0)}}  & 
  K_0 \bigl( C^\ast(G^{(1)}_{AFh}) \bigr) \ar@{->>}[r]^{\gamma_0^{(0)}}  & 
  K_0 \bigl( C^\ast(G^{res}_h) \bigr)
}
\end{equation}

We first study the exact sequences~\eqref{eq-ES-d1} and~\eqref{eq-ES-d0} to calculate the $K$-theory groups of $C^\ast(G_h^{dec})$ and $C^\ast(G^{res}_h)$ respectively.
We then substitute back into exact sequence~\eqref{eq-ES-d2} to compute the $K$-theory groups of $C^\ast(G)$.

\subsubsection*{$K$-theory groups of $C^\ast(G^{res}_h)$.}
We consider the exact sequence~\eqref{eq-ES-d0}.
By Corollary~\ref{prop-morita}, the groupoid $G^{res}_h$ is Morita equivalent to the groupoid of the $1$-dimensional substitution tiling of the horizontal edges.
Its $K_1$-goup is simply $\Z$, generated by the unitary of horizontal translation.
The $K$-groups of $G^{(0)}_{AF}$ and $G^{(1)}_{AFv}$ have been computed in Section~\ref{ssec-dimgp}, equations~\eqref{K-G0AF} and~\eqref{K-G1AFh}.
The exact sequence \eqref{eq-ES-d0} then reads:

\begin{equation}
\label{eq-ES-d0-explicit}
\xymatrix{ 
 \Z \ \ar@{^{(}->}[r]^{\alpha_0^{(0)}} &
 \Z^5 \ar@{->}[r]^(.37){\beta_0^{(0)}} &
 \Z\bigl[\frac{1}{2}\bigr] \oplus \Z^8\ar@{->>}[r]^(.42){\gamma_0^{(0)}} &
 K_0 \bigl( C^\ast(G^{res}_h) \bigr).
}
\end{equation}

As computed in Sections~\ref{ssec-dimgp} and~\ref{ssec-maps}, the generator of $K_1\bigl( C^\ast(G^{res}_h) \bigr) = \Z$ is sent to the constant projection $\ind_0 + \ind_1 \cdots + \ind_4\in \Z^5$, and $\coker \beta_0^{(0)}$ is torsion-free. 
Hence we have:

\begin{equation}
\label{K-H}
K_0\bigl( C^\ast(G^{res}_h ) \bigr) \simeq \Z\bigl[\frac{1}{2}\bigr] \oplus \Z^4, \qquad K_1\bigl( C^\ast(G^{res}_h) \bigr) \simeq  \Z.
\end{equation}

\subsubsection*{$K$-theory groups of $C^\ast(G_h^{dec})$.}
We consider the exact sequence~\eqref{eq-ES-d1}.
We computed in Section~\ref{ssec-dimgp} the $K$-theory groups of $G^{(1)}_{AFv}$ 
in equation~\eqref{K-G1AFv}, and the $K$-theory groups of  $G^{(2)}_{AF}$, in equation~\eqref{K-Gaf}.
Substituting these into the exact sequence~\eqref{eq-ES-d1} we get:
\[
\xymatrix{ 
 K_1 \bigl( C^\ast(G_h^{dec}) \bigr) \ar@{^{(}->}[r]^(.55){\alpha_0^{(1)}} & 
 \Z\bigl[\frac{1}{2}\bigr] \oplus \Z^8  \ar@{->}[r]^(.4){\beta_0^{(1)}}   &
 \Z\bigl[\frac{1}{4}\bigr] \oplus \Z\bigl[\frac{1}{2}\bigr]^2 \oplus \Z^{12} \ar@{->>}[r]^(.54){\gamma_0^{(1)}} &
 K_0 \bigl( C^\ast(G_h^{dec}) \bigr).
}
\]

The element $1$ in $\Z[1/2]$ in $K_0\bigl( C^\ast(G^{(1)}_{AFv}) \bigr)$ is the constant projection equal to one on all vertical edges, which is mapped to $0$ by $\beta_0^{(1)}$.
We computed in Section~\ref{ssec-maps} that the range of $\beta^{(1)}_0$ is free Abelian of rank $8$, and its cokernel is torsion-free.
It follows that 

\begin{equation}
\label{K-Gh}
K_0\bigl( C^\ast(G_h^{dec} ) \bigr) \simeq \Z\bigl[\frac{1}{4}\bigr] \oplus \Z\bigl[\frac{1}{2}\bigr]^2 \oplus \Z^4, \qquad K_1\bigl( C^\ast(G_h^{dec}) \bigr) \simeq  \Z\bigl[\frac{1}{2}\bigr].
\end{equation}

\subsubsection*{$K$-theory groups of $C^\ast(G)$.}
We substitute equations~\eqref{K-H} and~\eqref{K-Gh} into the exact sequence~\eqref{eq-ES-d2} to get
\[
\xymatrix{ 
\Z\bigl[\frac{1}{2}\bigr] \oplus \Z^4 \ar@{->}[r]^{\beta_0^{(2)}} & 
\Z\bigl[\frac{1}{4}\bigr] \oplus \Z\bigl[\frac{1}{2}\bigr]^2 \oplus \Z^4 \ar@{->}[r]^{\gamma_0^{(2)}} & 
K_0 \bigl( C^\ast(G) \bigr) \ar@{->}[d]^{\alpha_1^{(2)}} \\
K_1 \bigl( C^\ast(G) \bigr) \ar@{->}[u]^{\alpha_0^{(2)}} &  
\Z\bigl[\frac{1}{2}\bigr] \ar@{->}[l]^{\gamma_1^{(2)}}  & 
\Z \ar@{->}[l]^{\beta_1^{(2)}=0}
}
\]

We have the following properties
\begin{enumerate}
\item {\(\beta_0^{(2)} \Bigl( \Z \bigl[ \frac{1}{2}\bigr] \Bigr) =0\).} As above $\Z[1/2]$ is generated by the indicator function equal to $1$ on all horizontal edges, and is mapped to $0$ by $\beta_0$ as seen in Section~\ref{ssec-maps}.
\item \(\beta_0^{(2)} \bigl( \Z^4 \bigr)\) is identified with the \(\Z^4\) summand. This is the computation done in Section~\ref{ssec-maps}.
\item The generator of $\Z$ in $K_1(C^*(G^{(1)}_h))$ is the unitary of horizontal translation, and is mapped to $0$ by $\beta_1^{(2)}$. See Section~\ref{sec-PutES} after Theorem~\ref{thm-ESPut}.
\item The exact sequence splits in $\alpha_1^{(2)}$. Simply from the fact that $\Z$ is free.
\item From the splitting in $\alpha_1^{(2)}$ and what we established before, there is a short exact sequence
\[
 0 \rightarrow \Z[1/2] \rightarrow K_1 \bigl( C^\ast(G) \bigr) \rightarrow \Ker (\beta_0^{(2)}) \rightarrow 0,
\]
and the right-hand group is $\Z[1/2]$ by Section~\ref{ssec-maps}. A (left) splitting can be written by remarking that the unitary implementing horizontal translations $u_{(1,0)}$ is both an element of $C^\ast(G)$ and $C^\ast(G^{dec}_h)$. The map sending $u_{(1,0)}$ to itself produces a section of $\gamma_1^{(2)}$.
\end{enumerate}
It follows from 1, 2, 3, and 4, that

\begin{displaymath}
K_0 \bigl( C^\ast(G) \bigr)  \simeq \Z \bigl[ \frac{1}{4}\bigr] \oplus \Z \bigl[ \frac{1}{2}\bigr]^2 \oplus \Z,
\end{displaymath}
and from 1, 2, 3, and 5, that

\begin{displaymath}
K_1 \bigl( C^\ast(G) \bigr)  \simeq \Z \bigl[ \frac{1}{2}\bigr]^2.
\end{displaymath}

\subsubsection*{Acknowledgments}

We wish to thank Toke Carlsen and Ian Putnam for helpful discussions. An important part of this research was carried out in the CIRM (\emph{Centre international de rencontres math\'ematiques}) during a ``research in pairs'' in 2014.

\bibliographystyle{abbrv}
\bibliography{biblio-chair}

\begin{thebibliography}{10}

\bibitem{AP98}
J.~E. Anderson and I.~F. Putnam.
\newblock Topological invariants for substitution tilings and their associated
  {$C^*$}-algebras.
\newblock {\em Ergodic Theory Dynam. Systems}, 18(3):509--537, 1998.

\bibitem{BG13}
M.~Baake and U.~Grimm.
\newblock {\em Aperiodic order. {V}ol. 1}, volume 149 of {\em Encyclopedia of
  Mathematics and its Applications}.
\newblock Cambridge University Press, Cambridge, 2013.
\newblock A mathematical invitation, With a foreword by Roger Penrose.

\bibitem{BDHS10}
M.~Barge, B.~Diamond, J.~Hunton, and L.~Sadun.
\newblock Cohomology of substitution tiling spaces.
\newblock {\em Ergodic Theory Dynam. Systems}, 30(6):1607--1627, 2010.

\bibitem{BHZ00}
J.~Bellissard, D.~J.~L. Herrmann, and M.~Zarrouati.
\newblock Hulls of aperiodic solids and gap labeling theorems.
\newblock In {\em Directions in mathematical quasicrystals}, volume~13 of {\em
  CRM Monogr. Ser.}, pages 207--258. Amer. Math. Soc., Providence, RI, 2000.

\bibitem{BJS12}
J.~Bellissard, A.~Julien, and J.~Savinien.
\newblock Tiling groupoids and {B}ratteli diagrams.
\newblock {\em Ann. Henri Poincar\'e}, 11(1-2):69--99, 2010.

\bibitem{BESB94}
J.~Bellissard, A.~van Elst, and H.~Schulz-Baldes.
\newblock The noncommutative geometry of the quantum {H}all effect.
\newblock {\em J. Math. Phys.}, 35(10):5373--5451, 1994.
\newblock Topology and physics.

\bibitem{BM15}
M.~T. Benameur and V.~Mathai.
\newblock Gap-labelling conjecture with nonzero magnetic field.
\newblock {\em arXiv preprint arXiv:1508.01064}, 2015.

\bibitem{GHK13}
F.~G{\"a}hler, J.~Hunton, and J.~Kellendonk.
\newblock Integral cohomology of rational projection method patterns.
\newblock {\em Algebr. Geom. Topol.}, 13(3):1661--1708, 2013.

\bibitem{GPS04}
T.~Giordano, I.~Putnam, and C.~Skau.
\newblock Affable equivalence relations and orbit structure of {C}antor
  dynamical systems.
\newblock {\em Ergodic Theory Dynam. Systems}, 24(2):441--475, 2004.

\bibitem{JS12}
A.~Julien and J.~Savinien.
\newblock Tiling groupoids and {B}ratteli diagrams {II}: {S}tructure of the
  orbit equivalence relation.
\newblock {\em Ann. Henri Poincar\'e}, 13(2):297--332, 2012.

\bibitem{Kel97}
J.~Kellendonk.
\newblock The local structure of tilings and their integer group of
  coinvariants.
\newblock {\em Comm. Math. Phys.}, 187(1):115--157, 1997.

\bibitem{Kre16}
M.~Kreisel.
\newblock Gabor frames for quasicrystals, {$K$}-theory, and twisted gap
  labeling.
\newblock {\em Journal of Functional Analysis}, 270(3):1001--1030, 2016.

\bibitem{Mou10}
H.~Moustafa.
\newblock P{V} cohomology of the pinwheel tilings, their integer group of
  coinvariants and gap-labeling.
\newblock {\em Comm. Math. Phys.}, 298(2):369--405, 2010.

\bibitem{MRW87}
P.~S. Muhly, J.~N. Renault, and D.~P. Williams.
\newblock Equivalence and isomorphism for groupoid {$C^\ast$}-algebras.
\newblock {\em J. Operator Theory}, 17(1):3--22, 1987.

\bibitem{OOP11}
H.~Oyono-Oyono and S.~Petite.
\newblock {$C^*$}-algebras of {P}enrose hyperbolic tilings.
\newblock {\em J. Geom. Phys.}, 61(2):400--424, 2011.

\bibitem{Put89}
I.~F. Putnam.
\newblock The {$C^*$}-algebras associated with minimal homeomorphisms of the
  {C}antor set.
\newblock {\em Pacific J. Math.}, 136(2):329--353, 1989.

\bibitem{Put97}
I.~F. Putnam.
\newblock An excision theorem for the {$K$}-theory of {$C^*$}-algebras.
\newblock {\em J. Operator Theory}, 38(1):151--171, 1997.

\bibitem{Put98}
I.~F. Putnam.
\newblock On the {$K$}-theory of {$C^*$}-algebras of principal groupoids.
\newblock {\em Rocky Mountain J. Math.}, 28(4):1483--1518, 1998.

\bibitem{Sad08book}
L.~Sadun.
\newblock {\em Topology of tiling spaces}, volume~46 of {\em University Lecture
  Series}.
\newblock American Mathematical Society, Providence, RI, 2008.

\bibitem{Sol97}
B.~Solomyak.
\newblock Dynamics of self-similar tilings.
\newblock {\em Ergodic Theory Dynam. Systems}, 17(3):695--738, 1997.

\end{thebibliography}

\newpage

\appendix

\section{Computer use}
Parts of the computations were done using a computer algebra software. We provide here some of the matrices and explain the procedures used.
The computations should be easy to reproduce using any software able to compute eigenvalues, eigenvectors and the Schmidt normal form of integer matrices. They did not require any specialized library.

Note that there are probably more clever ways to approach these computations, using the action of $\Z/4\Z$ by rotation on the tiles. However, since computer assistance was needed anyway, the greedy approach was chosen.

First, define the matrix of the substitution on tiles.
This matrix $B^{(2)}$ is given as follows. The basis for $\Z^{24}$ is indexed by the six tiles of Figure~\ref{fig-tiles}, then their images by $r$, by $r^2$ etc.
\[
{\scriptsize
\left[
\begin{array}{cccccccccccccccccccccccc}
   2&1&1&1&0&0&0&0&0&0&0&0&0&0&0&0&0&0&0&0&0&0&0&0 \\
   0&1&0&0&1&1&0&0&0&0&0&0&0&0&0&0&0&0&0&0&0&0&0&0 \\
   0&0&1&0&1&0&0&0&0&0&0&0&0&0&0&0&0&0&0&0&0&0&0&0 \\
   0&0&0&1&0&1&0&0&0&0&0&0&0&0&0&0&0&0&0&0&0&0&0&0 \\
   0&0&0&0&0&0&0&0&0&0&0&0&0&0&0&0&0&0&1&1&1&1&1&1 \\
   0&0&0&0&0&0&1&1&1&1&1&1&0&0&0&0&0&0&0&0&0&0&0&0 \\
   0&0&0&0&0&0&2&1&1&1&0&0&0&0&0&0&0&0&0&0&0&0&0&0 \\
   0&0&0&0&0&0&0&1&0&0&1&1&0&0&0&0&0&0&0&0&0&0&0&0 \\
   0&0&0&0&0&0&0&0&1&0&1&0&0&0&0&0&0&0&0&0&0&0&0&0 \\
   0&0&0&0&0&0&0&0&0&1&0&1&0&0&0&0&0&0&0&0&0&0&0&0 \\
   1&1&1&1&1&1&0&0&0&0&0&0&0&0&0&0&0&0&0&0&0&0&0&0 \\
   0&0&0&0&0&0&0&0&0&0&0&0&1&1&1&1&1&1&0&0&0&0&0&0 \\
   0&0&0&0&0&0&0&0&0&0&0&0&2&1&1&1&0&0&0&0&0&0&0&0 \\
   0&0&0&0&0&0&0&0&0&0&0&0&0&1&0&0&1&1&0&0&0&0&0&0 \\
   0&0&0&0&0&0&0&0&0&0&0&0&0&0&1&0&1&0&0&0&0&0&0&0 \\
   0&0&0&0&0&0&0&0&0&0&0&0&0&0&0&1&0&1&0&0&0&0&0&0 \\
   0&0&0&0&0&0&1&1&1&1&1&1&0&0&0&0&0&0&0&0&0&0&0&0 \\
   0&0&0&0&0&0&0&0&0&0&0&0&0&0&0&0&0&0&1&1&1&1&1&1 \\
   0&0&0&0&0&0&0&0&0&0&0&0&0&0&0&0&0&0&2&1&1&1&0&0 \\
   0&0&0&0&0&0&0&0&0&0&0&0&0&0&0&0&0&0&0&1&0&0&1&1 \\
   0&0&0&0&0&0&0&0&0&0&0&0&0&0&0&0&0&0&0&0&1&0&1&0 \\
   0&0&0&0&0&0&0&0&0&0&0&0&0&0&0&0&0&0&0&0&0&1&0&1 \\
   0&0&0&0&0&0&0&0&0&0&0&0&1&1&1&1&1&1&0&0&0&0&0&0 \\
   1&1&1&1&1&1&0&0&0&0&0&0&0&0&0&0&0&0&0&0&0&0&0&0
  \end{array}\right]}.
\]
One checks that the eigenvalues and eigenvectors of $\big(B^{(2)}\big)^T$ are $4$, $2$, $1$ and $0$ with respective algebraic multiplicities $1$, $2$, $12$ and $9$. The geometric multiplicities are identical, except for the eigenvalue $0$, for which is just $8$.
A software can compute the eigenvectors of $B^{(2)}$, or check by a matrix-vector multiplication that the ones given in Section~\ref{sec-Kth} are indeed eigenvectors (for example, in this basis, $(r^2-1) \bigl(\ind_{A^0} + \ind_{B^0} + \ind_{C_1^0} + \ind_{C_2^0} + \ind_{D_1^0} + \ind_{D_2^0} \bigr)$ is the transpose of the row vector 
\[
[-1,-1,-1,-1,-1,-1,0,0,0,0,0,0,1,1,1,1,1,1,0,0,0,0,0,0],
\]
and is an eigenvector for the eigenvalue~$2$).

Similarly, the matrix for the substitution induced on the horizontal edges is
\[
 B^{(1)}_h = \begin{bmatrix}
   1 & 0 & 0 & 0 & 0 & 0 & 0 & 0 & 0 & 0 \\
   0 & 1 & 0 & 0 & 0 & 0 & 0 & 0 & 0 & 0 \\
   0 & 0 & 1 & 1 & 1 & 1 & 0 & 0 & 1 & 1 \\
   1 & 1 & 1 & 1 & 0 & 0 & 1 & 1 & 0 & 0 \\
   0 & 0 & 0 & 0 & 1 & 0 & 0 & 0 & 0 & 0 \\
   0 & 0 & 0 & 0 & 0 & 1 & 0 & 0 & 0 & 0 \\
   0 & 0 & 0 & 0 & 0 & 0 & 1 & 0 & 0 & 0 \\
   0 & 0 & 0 & 0 & 0 & 0 & 0 & 1 & 0 & 0 \\
   0 & 0 & 0 & 0 & 0 & 0 & 0 & 0 & 1 & 0 \\
   0 & 0 & 0 & 0 & 0 & 0 & 0 & 0 & 0 & 1 
 \end{bmatrix}
\]
in the basis indexed by the edges in order: $1a$, $1b$, etc. (See Figure~\ref{fig-edges}.) The eigenvalues are $2$, $1$ and $0$ with multiplicities $1$, $8$ and $1$ and the eigenvectors are the one given in Section~\ref{sec-Kth}.
For the substitution induced on the vertical edges, the matrix $B^{(1)}_v$ is identical in an appropriate basis (by convention we chose the basis indexed by the edges rotated by $\pi/4$).
The eigenvalues associates to eigenvalues $1$ are the ones corresponding to the edges $1a$, $1b$, $4a$, $4b$, $5$, $6$, $7$ and $8$, or taking advantage of the symmetry, they are the ones associated with the edges $a:=1a$, $b:=1b$, $d:=5$, $e:=6$ and their image under the symmetry $\rho$.

The computation of the eigenvalues of the transposes of these matrices is enough to give the $K$-theory of the $AF$-algebras.

Next, one computes the induced maps between the $K$-groups. For computing the image of $\beta_0^{(1)}$, it is enough to consider images of elements represented by vectors in the eventual range of $(B^{(1)})^T$ in $\Z^{10}$, \emph{i.e.} eigenvectors associated with eigenvalues $1$ and $2$.
We compute by hand that, in these bases, $\beta_0^{(1)}$ is given by a map $\Z^{10} \rightarrow \Z^{24}$ as follows
\begin{align*}
 \ind_{r.a^0} & \mapsto [1, 1, 0, 0, 0, 0, 0, 0, -1, 0, -1, 0, 0, 0, 0, 0, 0, 0, 0, 0, 0, 0, 0, 0]^T  \\
 \ind_{r.b^0} & \mapsto [0, 0, 0, 1, 0, 1, -1, -1, 0, 0, 0, 0, 0, 0, 0, 0, 0, 0, 0, 0, 0, 0, 0, 0]^T \\
 \ind_{r.d^0} & \mapsto [0, 0, 0, 0, 0, 0, 0, 0, 0, -1, 0, -1, 1, 0, 1, 1, 0, 0, 0, 0, 0, 0, 0, 0]^T \\
 \ind_{r.e^0} & \mapsto [0, 0, 1, 0, 1, 0, 0, 0, 0, 0, 0, 0, 0, 0, 0, 0, 0, 0, -1, 0, -1, -1, 0, 0]^T.
\end{align*}
Taking advantage of the equality $\beta \rho = -r^2 \beta$, one can obtain the image of these vectors with the opposite orientation. In the end, the map induced by $\beta^{(1)}_0$ on the eigenspace associated with $1$ is given by the following matrix $\Z^8 \rightarrow \Z^{24}$, where the basis for $\Z^8$ is indexed by $a^0, \rho a^0, b^0, \rho b^0, d^0, \ldots$ and the basis for $\Z^{24}$ is the same as above.
\[
\begin{footnotesize}
 \left[ \begin {array}{cccccccc}
  1&0&0&0&0&-1&0&0  \\
  1&0&0&0&0&0&0&0   \\ 
  0&0&0&0&0&-1&1&0  \\
  0&0&1&0&0&-1&0&0  \\
  0&0&0&0&0&0&1&0   \\
  0&0&1&0&0&0&0&0   \\
  0&0&-1&0&0&0&0&1  \\
  0&0&-1&0&0&0&0&0  \\
  -1&0&0&0&0&0&0&1  \\
  0&0&0&0&-1&0&0&1  \\
  -1&0&0&0&0&0&0&0  \\
  0&0&0&0&-1&0&0&0  \\
  0&-1&0&0&1&0&0&0  \\
  0&-1&0&0&0&0&0&0  \\
  0&0&0&0&1&0&0&-1  \\
  0&0&0&-1&1&0&0&0  \\
  0&0&0&0&0&0&0&-1  \\
  0&0&0&-1&0&0&0&0  \\
  0&0&0&1&0&0&-1&0  \\
  0&0&0&1&0&0&0&0   \\
  0&1&0&0&0&0&-1&0  \\
  0&0&0&0&0&1&-1&0  \\
  0&1&0&0&0&0&0&0   \\
  0&0&0&0&0&1&0&0
\end{array} \right].
\end{footnotesize}
\]
One checks by matrix-vector multiplication that all the columns of this matrix lie in the eigenspace of $(B^{(2)})^T$ associated with eigenvalue~$1$. Therefore, each of these images survives in the direct limit.
In addition, this matrix has a trivial kernel, which shows that the images not only survive in the limit, but they are linearly independent. 

One\footnote{A computer.} computes that image under $\beta$ of the eigenvector of $B^{(1)}$ associated with $2$ is $0$, and the eigenvector associated with $0$ is sent in the nullspace of $B^{(2)}$ (so that it is sent to $0$ in the limit).
Putting all this together, $\beta$ is determined by a map $\Z^{10} \rightarrow \Z^{24}$. There is a splitting (according to the invariant subspaces of the matrices):
\[
 \Z \oplus \Z \oplus \Z^8 \longrightarrow E_4 \oplus E_2 \oplus E_1 \oplus E_0,
\]
where the first two factors are sent to $E_0$ and $\Z^8$ is sent into $E_1$.
Now, a software can compute that the Schmidt normal form of the matrix above only has $1$ on the diagonal. Because of that, there is a basis of $E_1$ which splits it into $E_1 \simeq \Z^8 \oplus \Z^4$, such that the $\Z^8$ factor on the left is sent bijectively onto the $\Z^8$ factor on the right.
It shows that on the level of the $K$-groups, the cokernel of $\beta_0^{(1)}$ is torsion-free.

The computation for $\beta_0^{(2)}$ is essentially a repetition, taking into account the fact that elements in $K_0(C^*(G_h^{(1)}))$ can be represented by projections in the $AF$-subalgebra (since this group is a quotient), and the image in $K_0 (C^* (G^{dec}_h))$ can be computed by first computing the image in $K_0(C^*(G_AF^{(2)}))$ and doing then the appropriate identifications.

\end{document}